\documentclass[a4paper,reqno]{amsart}

\usepackage{amsfonts}
\usepackage{amsmath}

\usepackage[colorlinks]{hyperref}
\usepackage{amssymb}
\usepackage{tikz}
\usepackage{color}
\usepackage[all]{xy}
\usepackage{soul}
\usepackage{enumerate}
\usepackage{bm}
\usepackage{bbm}
\usepackage[normalem]{ulem}
\usepackage{mathrsfs}

\allowdisplaybreaks
\numberwithin{equation}{section}
\newtheorem{theorem}{Theorem}[section]
\newtheorem{proposition}[theorem]{Proposition}
\newtheorem{lemma}[theorem]{Lemma}
\newtheorem{corollary}[theorem]{Corollary}

\theoremstyle{definition}
\newtheorem{definition}[theorem]{Definition}

\theoremstyle{remark}
\newtheorem{remark}[theorem]{Remark}
\newtheorem{example}[theorem]{Example}

\newcommand{\NN}{\mathbb{N}}

\newcommand{\ZZ}{\mathbb{Z}}
\newcommand{\MCE}{\operatorname{MCE}}
\newcommand{\FE}{\operatorname{FE}}
\newcommand{\Ext}{\operatorname{Ext}}

\newcommand{\KP}{\operatorname{KP}}
\newcommand{\G}{\mathcal G}
\newcommand{\ep}{\varepsilon}
\newcommand{\g}{\gamma}
\newcommand{\ga}{\gamma}
\newcommand{\de}{\delta}
\newcommand{\Ga}{\Gamma}
\newcommand{\cod}{\bm{c}}
\newcommand{\dom}{\bm{d}}
\newcommand\Gr[1][]{{\operatorname{{\bf Gr}^{#1}-}}}
\newcommand{\id}{\operatorname{id}}
\newcommand{\supp}{\operatorname{supp}}
\newcommand\Modd[1][]{{\operatorname{{\bf Mod}^{#1}-}}}

\newcommand{\Span}{\operatorname{span}}
\newcommand{\gr}{\operatorname{gr}}

\title[Strongly graded groupoids and strongly graded Steinberg algebras]{Strongly graded groupoids\\ and strongly graded Steinberg algebras}

\author[L.O. Clark]{Lisa Orloff Clark}

\address{Lisa Orloff Clark: School of Mathematics and Statistics, 
Victoria University of Wellington, New Zealand}
\email{lisa.clark@vuw.ac.nz}

\author[R. Hazrat]{Roozbeh Hazrat}

\address{Roozbeh Hazrat: Centre for Research in Mathematics, 
Western Sydney University, Australia}
\email{r.hazrat@westernsydney.edu.au}

\author[S.W. Rigby]{Simon W. Rigby}

\address{Simon W. Rigby: Department of Mathematics and
 Applied Mathematics, University of Cape Town, South~Africa}

\email{rgbsim001@myuct.ac.za}

\begin{document}

\begin{abstract} 
	We study strongly graded groupoids, which are topological groupoids $\G$ equipped with a continuous, surjective functor $\kappa: \G \to \Gamma$, to a discrete group $\Gamma$, such that $\kappa^{-1}(\gamma)\kappa^{-1}(\delta) = \kappa^{-1}(\gamma \delta)$, for all $\gamma, \delta \in \Gamma$. We introduce the category of graded $\G$-sheaves, and prove an analogue of Dade's Theorem: $\G$ is strongly graded if and only if every graded $\G$-sheaf is induced by a $\G_{\ep}$-sheaf. The Steinberg algebra of a graded ample groupoid is graded, and we prove that the algebra is strongly graded if and only if the groupoid is. Applying this result, we obtain a complete graphical characterisation of strongly graded Leavitt path and Kumjian-Pask algebras.
\end{abstract}

\maketitle

\section{Introduction}

Graded rings possess a high degree of structure, or rigidity, that often compensates for otherwise exotic behaviour. A ring $A$ is \textit{graded} by a group $\Ga$ if it has a decomposition into additive subgroups $A = \bigoplus _{\ga \in \Ga}A_\ga$ such that $A_\ga A_\de \subseteq A_{\ga \de}$ for all $\gamma, \delta \in \Ga$. The ring is \textit{strongly graded} if 
$A_\gamma A_\delta=A_{\gamma\delta}$ for all $\gamma,\delta\in \Gamma$.
From our point of view, strong grading is the very best kind of grading. Strongly graded rings were studied extensively by 
Everett Dade. Among other theorems, he proved in \cite{dade1} that $A$ is strongly graded if and only if the category of graded $A$-modules is naturally equivalent to the category of $A_\ep$-modules (where $\ep$ is the identity in $\Gamma$). In other words, the $\ep$-component ``speaks" for the whole ring and carries some information about its other homogeneous components. If $A$ is strongly graded,  there is a one-to-one correspondence between the graded left ideals of $A$ and the left ideals of $A_\ep$ (see \cite[Remark 1.5.6]{rh}). It follows that $A$ satisfies the ascending/descending chain condition on graded left ideals if and only if $A_\varepsilon$ satisfies the ascending/descending chain conditions on left ideals. In a similar vein, a strongly graded ring is graded von Neumann regular (every homogeneous element has a von Neumann inverse) if and only if its $\ep$-component is von Neumann regular.

In this paper, we develop a theory of strongly graded \'etale groupoids. In many respects, it resembles the classical theory of strongly graded rings. A topological groupoid $\G$ is \textit{$\Gamma$-graded} (by a discrete group $\Gamma$) if there is a continuous functor $\kappa: \G \to \Gamma$. Equivalently, $\G$ is a disjoint union of clopen subsets $\G = \bigsqcup_{\gamma \in \Gamma} \G_\gamma$ such that $\G_\gamma \G_\delta \subseteq \G_{\gamma \delta}$ for all $\gamma, \delta \in \Gamma$. In analogy with graded rings, we say that $\G$ is \textit{strongly graded} if $\G_\gamma \G_\delta = \G_{\gamma \delta}$ for all $\gamma, \delta \in \Gamma$. It is much easier to work with gradings on groupoids than it is to work with gradings on rings and, in particular, it is much easier to verify if a groupoid is strongly graded. To illustrate this, in a graded ring $A = \bigoplus_{\gamma \in \Gamma} A_\gamma$, the set $A_\gamma A_\delta$ contains not only those elements of the form $a_\gamma a_\delta$, where $a_\gamma \in A_\gamma$ and $a_\delta \in A_\delta$, but also all finite sums of such elements. In contrast, if $\G = \bigsqcup_{\gamma \in \Gamma}\G_\gamma$ is a graded topological groupoid, then $\G_\gamma \G_\delta$ is nothing more than the set of elements of the form $g_\gamma g_\delta$, where $g_\gamma \in \G_\gamma$ and $g_\delta \in \G_\delta$. There are no sums to worry about. Moreover, every subgroupoid of a graded groupoid is graded, while not every subring of a graded ring is graded. Every element of a graded groupoid is homogeneous, while most elements in a graded ring are not homogeneous.

It is a common theme in mathematics that the theory of an algebraic structure benefits from an extrinsic approach. For instance, studying $G$-sets, on which a group $G$ acts, often illuminates the properties of $G$. 
Dade's Theorem shows that the ``internal" definition of strong grading for rings has an equivalent ``external" characterisation: one that takes meaning in the category of modules.
In many situations, a \textit{sheaf} is the appropriate kind of external structure on which an {\'etale groupoid} should act. That is, sheaves are to groupoids what modules are to rings. With this principle in mind, we define the category of \textit{graded $\G$-sheaves}, associated to a graded \'etale groupoid $\G$. In Theorem \ref{strgro}, we prove a groupoid-theoretic analogue of Dade's Theorem. It says that an ample groupoid is strongly graded if and only if every graded $\G$-sheaf is ``induced" by a $\G_\ep$-sheaf. This result gives an external characterisation of strong grading for groupoids.

While this theory is interesting in its own right, our main application is the study of strongly graded Steinberg algebras.
Steinberg algebras $A_R(\G)$ are convolution algebras of functions from an ample groupoid $\G$ to a commutative ring $R$. They first appeared independently in papers by  Steinberg \cite{steinberg1} and Clark, Farthing, Sims, and Tomforde \cite{CF}.
The primary motivation for constructing these algebras was to generalise other classes of algebras, especially inverse semigroup algebras and Leavitt path algebras.
Steinberg algebras provide a unifying theory and a new way of studying these seemingly disparate classes of algebras. For instance, various papers \cite{steinberg2016simplicity,steinberg2018chain,steinberg2018prime} have used Steinberg algebras to characterise, in terms of the underlying graph or inverse semigroup, when a Leavitt path algebra or inverse semigroup algebra is simple, (semi)prime, (semi)primitive, noetherian, or artinian. The Steinberg algebra model has also been put to use in \cite{clark2016using, steinberg1} to describe the centres of these algebras. Other classes of algebras that arise as Steinberg algebras include partial skew group rings associated to topological partial dynamical systems \cite{BG}, and the higher-rank analogues of Leavitt path algebras, known as Kumjian-Pask algebras \cite{CP}. Additionally, the theory of Steinberg algebras has succeeded in producing algebras with interesting prescribed properties, including the first examples of simple algebras of arbitrary Gelfand-Kirrilov dimension~\cite{nekrashevych2016growth}.
 
Graded ample groupoids produce graded Steinberg algebras, and many well-studied classes of Steinberg algebras (including Leavitt path algebras, Kumjian-Pask algebras, and partial skew group rings) receive a graded structure in this way \cite{CP,CS,hli}. We prove, in Theorem \ref{strong}, that the Steinberg algebra $A_R(\G)$ is strongly graded if and only if the groupoid $\G$ is strongly graded. This theorem is useful for applications, and it enables us to characterise the strong grading property for each of the classes mentioned above.
Moreover, we prove a graded version of Steinberg's Equivalence Theorem from \cite{steinberg2}, and use it to show that Theorems \ref{strgro} and \ref{strong} are equivalent in the sense that either one can be derived from the other.

We have a special interest in Leavitt path algebras, not least because they have provoked some of the most interesting recent developments in graded ring theory. Leavitt path algebras, first introduced in \cite{abrams05, ara07}, are $\ZZ$-graded algebras whose generators and relations are encoded in a directed graph. The construction is somewhat similar to (and indeed motivated by) the  graph $C^*$-algebra construction. It is not an exaggeration to say that the $\mathbb{Z}$-graded structure of Leavitt path algebras is what makes it possible to study them so successfully. For instance, in the very first paper on Leavitt path algebras, Abrams and Aranda Pino applied a decomposition into homogeneous components to establish a criterion for simplicity (see \cite[Theorem 3.11]{abrams05}). The same authors, in \cite{abrams08}, used graded ideals to prove that every Leavitt path algebra over a field is semiprime and semiprimitive. The graded Grothendieck group is a very important concept that emerged from the study of these $\ZZ$-graded algebras, and it is the most promising invariant in the ongoing classification programme for Leavitt path algebras (see~\cite[\S 7.3]{LPAbook} and \cite{rhann}).

Strong grading is especially important for Leavitt path algebras, because the 0-component of every Leavitt path algebra is ultramatricial (see \cite[Theorem 5.3]{ara07}). Many of the ``good" properties of those ultramatricial 0-components, like von Neumann regularity, are then passed to the other components via strong grading (when it is present). Strong grading has also been applied to calculate the graded Grothendieck groups of some Leavitt path algebras \cite[\S3.9.3]{rh}. Strongly graded Leavitt path algebras have even found an application in noncommutative algebraic geometry, where Paul Smith \cite{smith} uses them to give an expression of the quotient category of graded modules over a path algebra, modulo those that are the sum of their finite-dimensional submodules.

However, not every Leavitt path algebra is strongly graded, so it is valuable to understand which ones are. 
A characterisation of strongly graded Leavitt path algebras was known for graphs with finitely many edges and vertices: it is necessary and sufficient that every vertex connects to a cycle (see \cite[Theorem 3.15]{rhs}). Our new techniques lead quite easily to a complete characterisation of strongly graded Leavitt path algebras, for graphs of any size. We are able to prove a much more general result, characterising the strong grading property for $\ZZ^k$-graded Kumjian-Pask algebras of higher-rank graphs.

The structure of the paper moves from abstract to concrete. In Section \ref{prelim}, we extend some of the basic concepts in graded ring theory to the setting of graded rings with local units, among them Dade's Theorem (Theorem \ref{dadesthm}). We recall the notion of a graded groupoid $\G$ and its associated Steinberg $R$-algebra $A_R(\G)$, which is a graded algebra with local units. We recall the notion of a $\G$-sheaf, in preparation for introducing the category of graded $\G$-sheaves in Section~\ref{hgfhgfgf3}. In Section~\ref{hgfhgfgf3}, we prove the groupoid version of Dade's Theorem (Theorem~\ref{strgro}) and then we establish our main theorem that a Steinberg algebra $A_R(\G)$ is strongly graded if and only if $\G$ is so (Theorem~\ref{strong}).  We then establish an equivalence between the category of graded $\G$-sheaves of $R$-modules and the category of graded modules over the Steinberg $R$-algebra associated to $\G$. 
%
%
In Section \ref{applications}, we apply our results to graph algebras: both to Leavitt path algebras and to Kumjian-Pask 
algebras, their higher-rank analogues. We use the groupoid model to give graphical characterisations of strongly graded
graph algebras (Theorem~\ref{thm:LPA} and Theorem~\ref{thm:KP}).  As emphasised, this demonstrates why it is much easier to check that a groupoid is strongly graded than it is to check the corresponding criterion for the ring itself.  We also apply our results to partial group actions, showing that the algebra associated to a partial action of a discrete group on a totally disconnected space is strongly graded if and only if the action is a global action (Proposition~\ref{paction}).


\section{Preliminaries} \label{prelim}
A \textit{groupoid} is a small category in which every morphism has an inverse 
(see \cite[Definition 1.1]{renault} for an alternative definition). If $\G$ is a groupoid, 
we write $\G^{(0)}$ for the set of objects (called the \textit{unit space}) and we identify objects with their identity morphisms. 
We denote by $\dom$ and $\cod$ the \textit{domain} and \textit{codomain} maps $\dom, \cod: \G \to \G^{(0)}$ and write $\G^{(2)}\subseteq \G \times \G$ for the set of composable pairs, which are pairs $(x, y)$ with $\cod(y) = \dom(x)$. We also write $\G^x := \dom^{-1}(x)$ and $^x\G := \cod^{-1}(x)$.
A \textit{topological groupoid} is a groupoid $\G$ equipped with a topology such that 
inversion $\bm{i}: \G \to \G$ and composition $\bm{m}: \G^{(2)} \to \G$ are continuous, 
where $\G^{(2)}$ has the relative product topology. It is assumed that the unit space $\G^{(0)}$ has the relative topology (as a subspace of $\G$). In case $\dom$ is a local homeomorphism, 
we say that $\G$ is an \textit{\'etale groupoid}. If $\G$ is \'etale, it follows
that $\cod$ is also a local homeomorphism. If 
$U \subseteq \G$ is open and both $\dom|_{U}$ and $\cod|_{U}$ are injective 
then they are homeomorphisms onto their images, and $U$ is called an 
\textit{open bisection} (or a slice, or an open $\G$-set in some literature). 
The topology on an \'etale groupoid is generated by a basis of open bisections, 
and the unit space itself is an open bisection \cite[Proposition 3.2]{exel1}.

A topological groupoid $\G$ is called \textit{ample} if it has a basis of compact open bisections and $\G^{(0)}$ is 
Hausdorff. Equivalently, $\G$ is ample if it is \'etale and $\G^{(0)}$ is Hausdorff, locally compact, and 
totally disconnected \cite[Proposition 4.1]{exel2}. Let $B^{\rm co}(\G) = \{B \subseteq \G \mid B \text{ is a compact open bisection} \}$ and $B^{\rm co}(\G^{(0)}) = \{U \in B^{\rm co}(\G) \mid U \subseteq \G^{(0)}\}$.
In an ample groupoid $\G$, the set $B^{\rm co}(\G)$
is an inverse semigroup, with products defined as
$
UV = \big \{uv \mid (u, v) \in U\times V \cap \G^{(2)} \big \}
$
and inverses $U^{-1} = \{u^{-1} \mid u \in U\}$.

\subsection{Steinberg algebras}

Let $R$ be a unital commutative ring.  Following 
\cite{steinberg1},  one can construct an associative $R$-algebra from an ample groupoid. Let $\bm{1}_B: \G \to R$ be the characteristic function
of a compact open bisection $B \subseteq \G$.

\begin{definition} \label{st-def}
	Let $\G$ be an ample groupoid. Define $A_R(\G)$ as the $R$-submodule of $R^\G$ generated by the set	$\{\bm{1}_B  \mid B \in B^{\rm co}(\G)\}$.
	The \textit{convolution} of $f, g \in A_R(\G)$ is defined as 
	\begin{align*}
	f*g(x)= \sum_{\substack{y \in \G\\ \dom(y) = \dom(x)}}f(xy^{-1})g(y) = \sum_{\substack{(z, y)\in \G^{(2)} \\ zy  = x}} f(z)g(y) && \text{for all } x \in \G.
	\end{align*}
	The $R$-module $A_R(\G)$,  with the convolution, is called the \textit{Steinberg algebra} of $\G$ over $R$.
\end{definition}

On the generators, the convolution reduces to the formula 
$\bm{1}_B * \bm{1}_C = \bm{1}_{BC}$ for all $B, C  \in B^{\rm co}(\G)$. In particular, $\bm{1}_U* \bm{1}_V = \bm{1}_{U \cap V}$ whenever $U, V \in B^{\rm co}(\G^{(0)})$. Note that $A_R(\G)$ is unital if and only if $\G^{(0)}$ is compact, in which case $1 = \bm{1}_{\G^{(0)}}$ (see \cite[Propositon 4.11]{steinberg1}). If $\G$ is Hausdorff, $A_R(\G)$ equals the set of locally constant,
compactly supported $R$-valued functions on $\G$.


\subsection{Graded rings with graded local units}

Let $\Gamma$ be a group with identity $\varepsilon$. A ring $A$ (possibly without unit)
is called a \emph{$\Gamma$-graded ring} if $A=\bigoplus_{\gamma \in \Gamma} A_{\gamma}$,
where each $A_{\gamma}$ is an additive subgroup of $A$ and $A_{\gamma}  A_{\delta}
\subseteq A_{\gamma\delta}$,  for all $\gamma, \delta \in \Gamma$. By definition, $A_\gamma A_\delta$ is the additive subgroup generated by all terms $a_\gamma a_\delta$ where $a_\gamma \in A_\gamma$ and $a_{\delta} \in A_\delta$. The group $A_\gamma$ is
called the $\gamma$-\emph{homogeneous component} of $A$ (or sometimes just the \textit{$\gamma$-component}). The elements of $\bigcup_{\gamma \in \Gamma}A_\gamma$ are called \textit{homogeneous}, and the nonzero elements of $A_\gamma$ are called \textit{homogeneous of degree~$\gamma$}. If $a \in A$, we write $a = \sum_{\gamma \in \Gamma} a_\gamma$ for the unique expression of $a$ as a sum of homogeneous terms $a_\gamma \in A_\gamma$.  When it is clear from context
that the ring $A$ is graded by the group $\Gamma$, we simply say that $A$ is a \emph{graded
ring}. If $A$ is an $R$-algebra over a commutative unital ring $R$, then $A$ is called a \emph{graded algebra} if it 
is a graded ring and each $A_{\gamma}$ is an $R$-submodule of $A$. A \textit{graded homomorphism} of 
$\Gamma$-graded rings is a homomorphism $f: A \to B$ such that $f(A_\g) \subseteq B_\g$ for every $\g \in \Gamma$.
If a $\Gamma$-graded ring $A=\bigoplus_{\g\in\Ga}A_{\g}$ has the property that $A_{\gamma}A_{\delta}=A_{\gamma\delta}$ for all $\gamma,\delta\in \Ga$, then $A$ is called \emph{strongly $\Gamma$-graded}, 
or just \textit{strongly graded} if $\Gamma$ is clear from context.

We say that a ring $A$ has \emph{local units} if there is a set of idempotents $E \subseteq A$ and for every finite subset $F \subseteq A$ there exists $e \in E$ such that $exe = x$ for each $x \in F$. If $A$ is $\Gamma$-graded then we say $A$ has \textit{graded local units} if $E$ can be chosen as a subset of $A_\ep$.  (It appears to be unknown whether there exist graded rings with local units that are not graded rings with graded local units.) 


If $A$ is a $\Gamma$-graded ring and $\Omega$ is a normal subgroup of $\Gamma$, then it gives rise to two more graded rings: $A_\Omega := \bigoplus_{\omega \in \Omega}A_{\omega}$ is an $\Omega$-graded ring (with the \textit{subgroup grading}), and 
\begin{align}  \label{quot grad}
A = \bigoplus_{\Omega \gamma \in \Gamma / \Omega} A_{\Omega \gamma}, && \text{where }\ A_{\Omega \gamma} := \bigoplus_{\omega \in \Omega} A_{\omega \gamma} \text{ for all } \Omega \gamma\in \Gamma/\Omega,
\end{align}
is a $\Gamma/\Omega$-graded ring (with the \textit{quotient grading}).

Let $M$ be a right $A$-module. We say $M$ is \textit{unital} if $M=MA$. If $A$ is a $\Gamma$-graded ring, $M$ is called a \textit{graded $A$-module} if it has a decomposition $M=\bigoplus_{\g\in\Gamma}M_{\g}$ where each $M_\gamma$ is an additive subgroup of $M$ and
$M_{\gamma}A_\delta\subseteq M_{\g\delta}$ for all $\gamma,\delta \in \Gamma$. If $A$ has graded local units and $M$ is a unital graded $A$-module, then each $M_\gamma$ is a unital $A_\ep$-module. A graded homomorphism between graded $A$-modules is an $A$-module homomorphism $f: M \to N$ 
such that $f(M_\g) \subseteq N_\g$ for every $\g \in \Gamma$. We denote by $\Modd A$
the category of unital right $A$-modules and by $\Gr A$ the category of
unital graded right $A$-modules with graded homomorphisms.

For a graded right $A$-module $M$, and $\alpha \in \Gamma$, we define the $\alpha$-\emph{shifted} graded right
$A$-module $M(\alpha)$ as
\begin{equation}\label{eq:M-shifted}
M(\alpha)=\bigoplus_{\g\in \Ga}M(\alpha)_{\g},
\end{equation}
where $M(\alpha)_{\g}=M_{\alpha \gamma}$. That is, as an ungraded module, $M(\alpha)$ is a copy of
$M$, but the grading is shifted by $\alpha$. For $\alpha\in\Gamma$, the \emph{shift functor}
$
\mathcal{T}_{\alpha}: \Gr A\rightarrow \Gr A,$ $M\mapsto M(\alpha)$,
is an auto-equivalence with the property $\mathcal{T}_{\alpha}\mathcal{T}_{\beta}=\mathcal{T}_{\alpha\beta}$ for all $\alpha,\beta\in\Gamma$, and $\mathcal{T}_\ep = \mathrm{id}_{\Gr A}$.

The proof of the following lemma is straightforward. 

\begin{lemma} \label{first-lem}
	Let $A$ be a $\Gamma$-graded ring with graded local units. The following are equivalent: 
	
	\begin{enumerate}[\upshape(1)]
		\item $A$ is strongly graded.
		
		\smallskip

		\item $A_\gamma A_{\gamma^{-1}}=A_\varepsilon$ for every $\gamma \in \Gamma$.
		
		\smallskip

		\item For every $\g \in \Ga$, the set of graded local units is contained in $A_\gamma A_{\gamma^{-1}}$.
	\end{enumerate}
\end{lemma}

Next, we introduce the functors involved in Dade's Theorem.
Consider the  \emph{restriction functor}, mapping a graded $A$-module $M$ to the $A_\ep$-module $M_\ep$:
\begin{align}\label{funIri}
\mathcal I: \Gr A & \longrightarrow \Modd A_\ep\\
M & \longmapsto M_\ep \notag\\
\psi & \longmapsto \psi|_{M_\ep} \notag.
\end{align}
Consider the  \emph{induction functor}, mapping an $A_\ep$-module $N$ to the graded $A$-module $N \otimes_{A_\ep} A$:
\begin{align}\label{funIIri}
\mathcal J : \Modd A_\ep & \longrightarrow \Gr A \\
N & \longmapsto N \otimes_{A_\ep} A \notag\\
\phi & \longmapsto \phi \otimes \id .\notag
\end{align}
The grading on $N \otimes_{A_\ep} A$ is defined by setting $(N \otimes_{A_\ep} A)_\gamma = N \otimes_{A_\ep} A_\gamma$.
One can easily check that 
$\mathcal I \mathcal J \cong \id_{\Modd A_\ep}$ with the natural isomorphism:
\begin{align}\label{hgy4nd1}
\mathcal I \mathcal J (N)= \mathcal I (N\otimes_{A_\ep} A) =N\otimes_{A_\ep} A_\ep & \overset{\cong}{\longrightarrow} N,\\
n\otimes a &\longmapsto na  \notag.
\end{align}
On the other hand, there is a natural transformation:
\begin{align}\label{hgy4nd}
\mathcal J \mathcal I (M)= \mathcal J (M_\ep)= M_\ep\otimes_{A_\ep} A & \longrightarrow M,  \\* 
m\otimes a &\longmapsto ma  \notag.
\end{align}
The theorem below is Dade's Theorem in the setting of graded rings with graded local units.
The proof is similar to the case for unital rings, so we leave it to the reader (see~\cite[\S 1.5]{rh} and \cite{dade1}).

\begin{theorem}  \label{dadesthm} 
	Let $A$ be a $\Ga$\!-graded ring with graded local units. Then $A$ is strongly $\Gamma$\!-graded if and only if 
	the functors  $\mathcal I$ and $\mathcal J$ {\upshape (see (\ref{funIri}), (\ref{funIIri}))}  are mutually inverse equivalences of categories. 
\end{theorem}

Note that it is possible to have an equivalence between the categories $\Gr A$ and $\Modd A_\ep$, without $A$ being strongly graded; see \cite[Example 3.2.4]{nasta}. Hence, the functors $\mathcal{I}$ and $\mathcal{J}$ are an essential part of Dade's Theorem.

Applying Dade's Theorem twice we can prove the following lemma. We give an element-wise proof of the lemma, as this paves the way for the proof of the groupoid version of it (see Lemma~\ref{first-lemgrgr}). 

\begin{lemma} \label{corr}
	Let $A$ be a $\Ga$-graded ring with graded local units, and let $\Omega \lhd \Gamma$. Then $A$ is strongly $\Ga$-graded if and only if $A$ is strongly $\Ga/\Omega$-graded and $A_\Omega$ is strongly $\Omega$-graded.
\end{lemma}

\begin{proof} Suppose $A$ is a strongly $\Ga$-graded ring. By Lemma \ref{first-lem},  $A_\gamma A_{\gamma^{-1}}=A_\varepsilon$, for every $\gamma \in \Gamma$. Restricting to $\Omega$, we immediately get that $A_\Omega$ is strongly $\Omega$-graded. Next we show that $A_{\Omega \gamma}A_{\Omega \gamma^{-1}}=A_{\Omega}$, for every $\gamma \in \Gamma$. The fact $A_{\Omega \gamma}A_{\Omega \gamma^{-1}} \subseteq A_{\Omega}$ follows from the definition of the grading. It is enough to show that for $\omega \in \Omega$, $A_\omega \subseteq A_{\Omega \gamma}A_{\Omega \gamma^{-1}}$. But 
$A_\omega=A_\omega A_\varepsilon=A_\omega A_\gamma A_{\gamma^{-1}} \subseteq A_{\Omega \gamma}A_{\Omega \gamma^{-1}}$. Applying Lemma~\ref{first-lem} (for $\Gamma/\Omega$-graded ring $A$) it follows $A$ is strongly $\Gamma/\Omega$-graded. 

For the converse, we only need to show that $A_\varepsilon \subseteq A_\gamma A_{\gamma^{-1}}$, for every $\gamma \in \Gamma$. If $\gamma \in \Omega$ then $A_\varepsilon= A_\gamma A_{\gamma^{-1}}$, because $A_\Omega$ is strongly graded. Suppose $\gamma \not \in \Omega$. 
Since $A$ is strongly $\Gamma/\Omega$-graded, we have $A_\varepsilon \subseteq A_\Omega= A_{\Omega\gamma}  A_{\Omega\gamma^{-1}}$. Thus for $x\in A_\varepsilon$ we have $x=\sum a_i b_i$, where $a_i \in  A_{\omega_{i_1}\gamma}$ and $b_i \in A_{\omega_{i_2}\gamma^{-1}}$. Comparing the degrees of both sides we have that $\deg(a_i)=\gamma \omega_i$ and $\deg(b_i)=\omega_i^{-1}\gamma^{-1}$, for $\omega_i \in \Omega$. Now write $a_i=a_iu_i$, where $u_i\in A_\varepsilon$ is a local unit. Since $A_\Omega$ is strongly $\Omega$-graded, by Lemma~\ref{first-lem}, $u_i \in A_{\omega_i^{-1}} A_{\omega_i}$.  Writing $x=\sum a_i b_i=\sum a_i u_ib_i$ and replacing $u_i$ from above, we get $x\in A_\gamma A_{\gamma^{-1}}$. By Lemma~\ref{first-lem}, this proves that $A$ is strongly $\Gamma$-graded.
\end{proof}

\subsection{Graded groupoids}

Let $\Gamma$ be a group (with identity $\varepsilon$) and let $\G$ be a topological groupoid. 
The groupoid $\G$ is called \textit{$\Gamma$-graded} if $\G$ can be partitioned by clopen subsets indexed by 
$\Ga$, i.e.\! $\G = \bigsqcup_{\gamma \in \Gamma} \G_\gamma$, 
such that $\G_\gamma \G_\delta \subseteq \G_{\gamma \delta}$ for 
every $\gamma, \delta \in \Gamma$. The set $\G_\gamma$ is called the \textit{$\gamma$-component} of $\G$.
We write
$\G_\gamma^x := \dom^{-1}(x) \cap \G_\gamma$ and $^x\G_\gamma:= \cod^{-1}(x) \cap \G_\gamma$. We say a subset $X\subseteq \G$ is 
\textit{$\gamma$-homogeneous} if $X\subseteq \G_\gamma$. 
Obviously, the unit space is $\varepsilon$-homogeneous and if 
$X$ is $\gamma$-homogeneous then $X^{-1}$ is $\gamma^{-1}$-homogeneous. 

Equivalently, $\G$ is $\Gamma$-graded if there is a continuous functor 
$\kappa: \G \to \Gamma$, where $\Gamma$ is regarded as a discrete group. To match the definition of $\Gamma$-grading, from the previous paragraph, one defines $\G_\gamma = \kappa^{-1}(\gamma)$. We say that the graded groupoid $\G$ is \textit{strongly graded} 
if $\G_\gamma \G_\delta = \G_{\gamma \delta}$ for every $\gamma, \delta \in \G$. If the grading is defined by $\kappa: \G \to \Gamma$, then $\G$ is strongly graded if $\kappa^{-1}(\gamma) \kappa^{-1}(\delta) = \kappa^{-1}(\gamma \delta)$ for all $\gamma, \delta \in \Gamma$. Strong grading implies $\kappa$ is surjective.
Strongly graded groupoids  appeared in \cite[Definition~5.3.7]{AD} where they are
viewed as groupoids with a ``strongly surjective" functor $\kappa: \G \to \Gamma$.

If $\G$ is a $\Gamma$-graded topological groupoid and $\Omega \lhd \Gamma$, then it gives rise to two more graded groupoids. Firstly, the open subgroupoid $\G_{\Omega}:=\bigsqcup_{\omega \in \Omega} \G_\omega$ is an 
$\Omega$-graded groupoid (with the \textit{subgroup grading}).
Secondly, if we view $\Gamma$ as a discrete group,  then the quotient topology on $\Gamma/\Omega$ is also discrete, so 
$\G$ has a natural $\Gamma/\Omega$-grading (the \textit{quotient grading}) defined by 
\begin{align} \label{quot grad 2}
\G=\bigsqcup_{\Omega \gamma \in \Gamma/\Omega} \G_{\Omega \gamma},&&
\text{where } \G_{\Omega \gamma}:=\bigsqcup_{\omega\in \Omega} \G_{\omega\gamma} \text{ for all } \Omega \gamma \in \Gamma/\Omega.
\end{align}

\subsection{Graded Steinberg algebras}
Graded ample groupoids are very important, because the Steinberg algebra that one gets from such a groupoid is always a graded algebra. This fact is essential for a number of key results: for instance, in~\cite{CS} it is used to prove that every Leavitt path algebra is a Steinberg algebra, and in \cite{hli} it is used to prove that Steinberg algebras of certain transformation groupoids are partial skew group rings.

If an ample groupoid $\G$ is $\Gamma$-graded then we write $B^{\rm co}_{\gamma}(\G)$ for the set of all $\gamma$-homogeneous compact open bisections of $\G$, and note that the set of all homogeneous compact open bisections,
$
B_{*}^{\rm co}(\G):=\bigcup_{\gamma\in\Gamma} B^{\rm co}_{\gamma}(\G)$, is a basis for the 
topology on $\G$. Recall that $A_R(\G)$ is defined in Definition \ref{st-def}.

\begin{lemma} \cite[Lemma 3.1]{CS}
	If $\G$ is a $\Gamma$-graded ample groupoid, then $A_R(\G) = \bigoplus_{\gamma \in \Gamma} A_R(\G)_\gamma$ is a $\Gamma$-graded algebra with homogeneous components:
	\begin{equation*}
		A_R(\G)_\gamma = \{f \in A_R(\G) \mid \supp(f) \subseteq \G_\gamma\} = \Span_R\{\bm{1}_B \mid B \in B_\gamma^{\rm co}\}.
	\end{equation*} 
\end{lemma}
	
Naturally, $A_R(\G)_\ep \cong A_R(\G_\ep)$ via the isomorphism $f \mapsto f|_{\G_\ep}$. If $\G$ is $\Gamma$-graded then $A_R(\G)$ has graded local units, which are the characteristic 
functions of compact open subsets of $\G^{(0)}$
(see \cite[Lemma~2.6]{CEP}).
The quotient grading and subgroup grading are preserved  by the construction of Steinberg algebras. Specifically, if $\G$ is $\Gamma$-graded and $\Omega \lhd \Gamma$, the $\Gamma/\Omega$-graded structure 
on $A_R(\G)$ can be obtained in the following two equivalent ways: either by viewing  $\G$ as a $\Gamma/\Omega$-graded groupoid, as in (\ref{quot grad 2}), and transferring that grading to $A_R(\G)$, or by giving $A_R(\G)$ the usual $\Gamma$-graded structure and taking the quotient grading of $A_R(\G)$, as in (\ref{quot grad}). This is due to the fact (whose proof is similar to \cite[Lemma 3.1]{CS}) that \[\bigoplus_{\omega \in \Omega} \Big\{f \in A_R(\G) \mid \supp(f) \subseteq \G_{\omega \gamma} \Big\} = \Big\{f \in A_R(\G)\mid \supp(f) \subseteq \bigsqcup_{\omega \in \Omega} \G_{\omega \gamma}\Big\}
\]
for all $\Omega \gamma \in \Gamma /\Omega$. Similarly, $A_R(\G_\Omega)$ is graded isomorphic to $A_R(\G)_\Omega$, because
\[
A_R(\G_\Omega) \cong \{f \in A_R(\G) \mid \supp(f) \subseteq \G_\Omega\} = \bigoplus_{\omega \in \Omega}\{f \in A_R(\G) \mid \supp(f) \subseteq \G_{\omega}\} = A_R(\G)_\Omega.
\]

\subsection{$\G$-sheaves}
Let $X$ be a topological space. A \textit{sheaf space} over $X$ is a pair $(E,p)$ where $E$ is a topological space and $p: E \to X$ is a local homeomorphism. If $(E,p)$ is a sheaf space, the fibre $E_x:= p^{-1}(x)$ is called the \textit{stalk} of $E$ at $x \in X$. A \textit{local section} of $E$ is a map $s: U \to E$, where $U \subseteq X$ is open, such that $ps$ is the identity on $U$. A \textit{global section} of $E$ is a section $s: X \to E$. The set $\{s(U) \mid s: U \to E \text{ is a continuous local section}\}$ is a basis that generates the topology on $E$.  A \textit{morphism of sheaf spaces} $(E,p)$ and $(F,q)$ over $X$ is a continuous map $\phi: E \to F$ such that $q\phi = p$. In practice, a sheaf space $(E,p)$ is often referred to by $E$ when there is no need to draw attention to $p$. The following lemma is extremely useful.
  \begin{lemma}\cite[\S2, Lemma 3.5 (c)]{tennison} \label{useful}
 	If $(E,p)$ and $(F,q)$ are sheaf spaces over $X$ and $\phi: E \to F$ satisfies $q\phi = p$, then the following are equivalent:
 	\begin{enumerate}[\rm (1)]
 		\item $\phi$ is continuous;
 		\item $\phi$ is an open map;
 		\item $\phi$ is a local homeomorphism.
 	\end{enumerate}
 \end{lemma}

A sheaf space $(E,p)$ over $X$ is called a \textit{sheaf of $R$-modules} (where $R$ is a commutative ring with 1) provided each stalk $E_x$ is an $R$-module and the following conditions are satisfied:
\begin{enumerate}[({A}1)]
	\item The zero section $Z: X \to E$ sending $x \in X$ to $0_x$ (the zero of $E_x$) is continuous;
	\item Addition $E \times_{\G^{(0)}} E$ is continuous (where the fibre product is with respect to $p$);
	\item Scalar multiplication $R \times E \to E$ is continuous (where $R$ has the discrete topology).
\end{enumerate}
One can use Lemma \ref{useful} to show that addition and scalar multiplication are also open maps.
A \textit{morphism of sheaves of $R$-modules} is a morphism of sheaf spaces that restricts to $R$-homomorphisms on the stalks.  A section $s: X \to E$ is called \textit{compactly supported} if $\supp(s) = \{x \in X \mid s(x) \ne 0_x\}$ is compact. If $s$ is a continuous section, $\supp(s) = X \setminus Z^{-1}(s(X))$ is closed in $X$. Condition (A1) is equivalent to the statement that $\{0_x \mid x \in X\}$ is open in $E$.

Let $\mathfrak{C}$ be either the category of sets or the category of $R$-modules. Let $X$ be a topological space and $\mathcal{O}_X$ the poset of open subsets of $X$. A \textit{\textit{presheaf} over $X$} is a contravariant functor $F: \mathcal{O}_X \to \mathfrak{C}$. In practice, it is sufficient to define a presheaf $F$ on a subset $\mathcal{O}_X'$ of $\mathcal{O}_X$, provided $\mathcal{O}_X'$ is a basis for the topology on $X$. More precisely, $F$ is a \textit{presheaf of sets} or a \textit{presheaf of $R$-modules}, according to the category $\mathfrak{C}$. A \textit{morphism of presheaves} $F: \mathcal{O}_X \to \mathfrak{C}$ and $G: \mathcal{O}_X \to \mathfrak{C}$ is a natural transformation $\nu: F \to G$. 
Given any presheaf of sets $F$, a standard process converts $F$ into a sheaf space $(\widetilde{F}, p_F)$ in a functorial way (and the same process converts a presheaf of $R$-modules into a sheaf of $R$-modules). For every $U \in \mathcal{O}_X$ and every $x \in U$ there is a surjective homomorphism $F(U) \to \widetilde{F}_x$ usually denoted by $s \mapsto [s]_x$. Restricted to a subcategory of presheaves that satisfy two extra conditions (see \cite[p. 14]{tennison}), the functor $F \mapsto (\widetilde{F}, p_F)$ is an equivalence of categories. We refer to \cite{tennison} for the details.

Now let $\G$ be an \'etale groupoid. A \textit{(right) $\G$-sheaf} consists of a sheaf space $(E,p)$ over $\G^{(0)}$ together with a continuous action $E \times_{\G^{(0)}} \G \to E$ (where the fibre product is with respect to $p$ and $\cod$), denoted $(e,g) \mapsto eg$, satisfying the conditions:
\begin{enumerate}[({B}1)]
	\item $ep(e) = e$ for all $e \in E$;
	\item $p(eg) = \dom(g)$ whenever $e \in E$, $g \in \G$, and $p(e) = \cod(g)$;
	\item $(eg)h = e(gh)$ whenever $e \in E$, $(g, h) \in \G^{(2)}$, and $p(e) = \dom(g)$.
\end{enumerate}
If $E$ and $F$ are $\G$-sheaves, a morphism of sheaf spaces $\phi: E \to F$ is called \textit{$\G$-equivariant} if $\phi(eg) = \phi(e)g$ for all $(e,g) \in E \times_{\G^{(0)}} \G$. A \textit{morphism of $\G$-sheaves} is a $\G$-equivariant morphism of sheaf spaces. The category of $\G$-sheaves, denoted $\mathcal{B}\G$, is called the \textit{classifying topos} of $\G$ \cite{moerdijkI, moerdijkII}.

A $\G$-sheaf $(E,p)$ is called a \textit{$\G$-sheaf of $R$-modules} if it is a sheaf of $R$-modules over $\G^{(0)}$ and for each $g \in \G$ the map $R_g: E_{\cod(g)} \to E_{\dom(g)}$, given by $R_g(e) = eg$, is an $R$-homomorphism. A \textit{morphism of $\G$-sheaves of $R$-modules} is a $\G$-equivariant morphism of sheaves of $R$-modules.
We write $\Modd_{R}\G$ for the category of $\G$-sheaves of $R$-modules.

\subsection{Steinberg's Equivalence Theorem} \label{SET}

The main reason we are interested in $\G$-sheaves of $R$-modules is that they are equivalent to representations of Steinberg algebras. In \cite[Theorem 3.5]{steinberg2}, Steinberg proved that (for ample $\G$) the category of $\G$-sheaves of $R$-modules is equivalent to the category of right unital $A_R(\G)$-modules. This is a vast generalisation of the fact that the category of representations of a group $G$ in $K$-vector spaces is equivalent to the category of $KG$-modules. Steinberg's Equivalence Theorem (as we call it) was used in \cite{steinberg2016simplicity} to study primitive Steinberg algebras, and it leads to a very short proof of the Primitivity Theorem for Leavitt path algebras (see \cite[Theorem 5.5]{steinberg2018prime} and \cite[Theorem 5.7]{ABR}). We briefly describe the theorem and the functors involved in it.

Let $M$ be a right $A_R(\G)$-module. Define the following presheaf of $R$-modules over $\G^{(0)}$:
\begin{align}\label{MU}
	&{M}(U) = M \bm{1}_U, & && &\text{for all } U \in B^{\rm co}(\G^{(0)});\\ \notag
	&\rho^U_V: M(U) \to M(V), & &\rho^U_V(m) = m \bm{1}_V, &&\text{for all } V \subseteq U \text{ in } B^{\rm co}(\G^{(0)}),\ m \in M(U).
\end{align}
Applying the sheaf space functor, one gets a sheaf of $R$-modules $\operatorname{Sh}(M) = (\widetilde{M},p_M)$ where:
\begin{align*}
	&\widetilde{M} = \bigcup_{x \in \G^{(0)}} \widetilde{M}_x, & &\widetilde{M}_x = \lim_{\substack{\longrightarrow\\x \in U}} M(U) = \{[m]_x \mid m \in M\}; \\
	&p_M([m]_x) = x & &\text{for all } [m]_x \in \widetilde{M}_x.
\end{align*}
In other words, $\widetilde{M}_x$ is the direct limit of the directed system in (\ref{MU}), as $U$ ranges over all the sets $U \in B^{\rm co}(\G^{(0)})$ that contain $x$. The notation $[m]_x$ is standard: it means the image of $m \in M$ in the direct limit $\widetilde{M}_x$. The topology on $\widetilde{M}$ is generated by the basis of open sets:
\begin{align}\label{Um}
&(U,m) = \big\{[m]_x \mid x \in U\big \},  & & \text{where } U \in B^{\rm co} (\G^{(0)}),\ m \in M(U).
\end{align}
For $[m]_x \in \widetilde{M}$ and $g \in {^x\G^y}$, we define $[m]_{x} g = [m \bm{1}_B]_{y}$ where $B \in B^{\rm co}(\G)$ is an arbitrarily chosen compact open bisection containing $g$. This makes $\widetilde{M}$ a $\G$-sheaf of $R$-modules.
After defining the effect of $\operatorname{Sh}$ on morphisms and checking some details, this process defines a functor $\operatorname{Sh}: \Modd A_R(\G) \to \Modd_R\G$.

In the other direction, let $(E,p)$ be a $\G$-sheaf of $R$-modules. Define $\Gamma_c(E)$ as the set of compactly supported continuous sections $s: \G^{(0)} \to E$. Then $\Gamma_c(E)$ is an $R$-module and, moreover, it is a right $A_R(\G)$-module with:
\begin{align}\label{ARG-mod}
\big(sf\big)(x) = \sum_{g \in \G^x} f\big(g\big)s\big(\cod(g)\big)g && \text{for all } s \in \Gamma_c(E),\ f \in A_R(\G),\ x \in \G^{(0)}.
\end{align}
In particular, (\ref{ARG-mod}) entails:
\begin{align*}
\big(s  \bm{1}_B\big)(x)=\begin{cases}  s\big(\cod(g)\big)g & \text{if}\ g\in B,\ \dom(g)=x\\ 0_x  & \text{otherwise}.\end{cases} && \text{for all } s \in \Gamma_c(E),\ B \in B^{\rm co}(\G),\ x \in \G^{(0)}.
\end{align*}
If $U \in B^{\rm co}(\G^{(0)})$ then $\big(s  \bm{1}_U\big)(x) = \bm{1}_U(x)s(x)$ for all $x \in \G^{(0)}$.
After defining the effect of $\Gamma_c$ on morphisms and checking some details, this defines a functor $\Gamma_c: \Modd_R(\G) \to \Modd A_R(\G)$. Here is the statement of Steinberg's Equivalence Theorem:

\begin{theorem} \cite[Theorem 3.5]{steinberg2}
	Let $\G$ be an ample groupoid and $R$ a unital commutative ring. The functors $\Gamma_c: \Modd_R\G \to \Modd A_R(\G)$
	and $\operatorname{Sh}: \Modd A_R(\G) \to \Modd_R\G$
	are mutually inverse equivalences of categories.
\end{theorem}

\begin{example} \label{first-ex}
	Consider the right regular representation $M := A_R(\G)$.
	The sheaf of $R$-modules $\operatorname{Sh}(M) = (\widetilde{M}, p_M)$ has stalks $\widetilde{M}_x  = \lim_{\substack{\longrightarrow \\ x \in U}} A_R(\G)*\bm{1}_U\cong R{\G^x}$, where $R{\G^x}$ is the free $R$-module with basis $\G^x$. The isomorphism $\big(\lim_{\substack{\longrightarrow \\ x \in U}} A_R(\G\big)*\bm{1}_U)\to R{\G^x}$ carries $[f]_x$ to $\sum_{g \in \G^x}f(g)g$. The induced action of $\G$ on the sheaf $\bigcup_{x \in \G^{(0)}} R\G^x$ is the canonical one where $(g,h) \mapsto gh$ for all $(g, h) \in \G^{(2)}$. Since $R\G^x$ is a limit of left $A_R(\G)$-modules, and the connecting homomorphisms are $A_R(\G)$-homomorphisms, $R\G^x$ is also a left $A_R(\G)$-module with \begin{align} \label{RGx-action}
	f \cdot t = \sum_{z \in \G^x} f(zt^{-1})z && \text{for all } f \in A_R(\G),\ t \in \G^x.
	\end{align}
	In particular, if $U \in B^{\rm co}(\G^{(0)})$ then $\bm{1}_U \cdot t = \bm{1}_U(\cod(t))t$. More details on $R\G^x$ as a bimodule can be found in \cite[Proposition 7.8]{steinberg1}.
\end{example}

\section{Strongly graded groupoids}\label{hgfhgfgf3}

First, we prove groupoid versions of Lemmas \ref{first-lem} and \ref{corr}. It is the first of several instances where graded groupoids and graded rings display surprisingly similar results.

\begin{lemma}\label{stg}
	Let $\G$ be a $\Gamma$-graded topological groupoid. The following are equivalent.
	\begin{enumerate}[\upshape(1)]
		\item\label{it1:stg} $\G$ is strongly graded;
		
		\smallskip

		\item\label{it2:stg} $\G_\gamma \G_{\gamma^{-1}}=\G_\varepsilon$, for all $\gamma \in \Gamma$;
		
		\smallskip

		\item\label{it3:stg} $\dom(\G_\gamma)=\G^{(0)}$, for all $\gamma \in \Gamma$;
		
		\smallskip

		\item\label{it4:stg} $\cod(\G_\gamma)=\G^{(0)}$, for all $\gamma \in \Gamma$.
	\end{enumerate}
\end{lemma}
\begin{proof}
	
	\noindent\eqref{it1:stg} $\Rightarrow$ \eqref{it2:stg} It follows from the definition of strong grading.
	
	\noindent \eqref{it2:stg} $\Rightarrow$ \eqref{it3:stg} For any $\gamma \in \Gamma$, 
	we have $ \G^{(0)} = \dom(\G_\varepsilon)=\dom(\G_\gamma \G_{\gamma^{-1}} )\subseteq \dom(\G_{\gamma^{-1}}) \subseteq \G^{(0)}$.
	
	\noindent\eqref{it3:stg} $\Rightarrow$ \eqref{it4:stg} 
	For any $\gamma \in \Gamma$, $\G^{(0)} = \dom(\G_{\gamma^{-1}}) = \cod \bm{i}(\G_{\gamma^{-1}}) = \cod(\G_\g)$.
	
	\noindent\eqref{it4:stg} $\Rightarrow$ \eqref{it1:stg} For any $x\in \G_{\gamma \delta}$ 
	choose $y\in  \G_{\delta^{-1}}$ with $\cod(y)=\dom(x)$. Then  $x=xyy^{-1} \in \G_\gamma \G_\delta$. 
\end{proof}

\begin{lemma}\label{first-lemgrgr}
	Let $\G$ be a $\Gamma$-graded topological groupoid and $\Omega\lhd \Gamma$. Then $\G$ is strongly $\Gamma$-graded if and only if $\G$ is strongly 
	$\Gamma/\Omega$-graded and $\G_{\Omega}$ is strongly $\Omega$-graded. 
\end{lemma}

\begin{proof} ($\Rightarrow$)
	If $\G$ is strongly $\Gamma$-graded then Lemma~\ref{stg} (2)
	implies  $\G$ is strongly $\Gamma/\Omega$-graded and $\G_{\Omega}$ is strongly $\Omega$-graded. 
	
	($\Leftarrow$)  By Lemma~\ref{stg} (3), $\dom(\G_\omega)=\G^{(0)}$ for any $\omega \in \Omega$, since $\G_\Omega$ is strongly $\Omega$-graded. Suppose $\gamma \in \Gamma \setminus \Omega$. 
	Since $\G$ is strongly $\Gamma/\Omega$-graded, again by Lemma~\ref{stg} (3) we have $\dom(\G_{\Omega \gamma})=\G^{(0)}$. Then for any $u \in \G^{(0)}$, there exists $\omega \in \Omega$ and $g\in \G_{\omega\gamma}$ such that $\dom(g)=u$. Now $\cod(g)\in \G^{(0)}=\dom(\G_{\omega^{-1}})$ so there exists $h\in \G_{\omega^{-1}}$ such that $\dom(h)=\cod(g)$. Then $hg\in \G_{\omega^{-1}}\G_{\omega\gamma}\subseteq \G_\gamma$ with $\dom(hg)=u$. Thus $u \in \dom(\G_\gamma)$ and so $\dom(\G_\gamma)=\G^{(0)}$. By Lemma~\ref{stg}, it follows that $\G$ is strongly $\Gamma$-graded. 
\end{proof}

\subsection{The category of graded $\G$-sheaves} \label{sheaves}

To motivate Definition \ref{G-sheaf of sets}, let $\Gamma$ be a group and let $X$ be a topological space. Let $\operatorname{Set}_\Gamma$ be the category whose objects are $\Gamma$-graded sets (that is, sets $Y$ equipped with a function $\kappa_Y: Y \to \Gamma$) and whose morphisms are functions $f: Y_1 \to Y_2$ such that $\kappa_{Y_2} f = \kappa_{Y_1}$. A $\operatorname{Set}_\Gamma$-valued presheaf over $X$ is a contravariant functor $F: \mathcal{O}_X \to \operatorname{Set}_\Gamma$. Applying the sheaf space functor to $F$, and keeping track of the $\Gamma$-grading, one obtains a sheaf space $(\widetilde{F},p_F)$ whose stalks are $\Gamma$-graded sets. If, in addition, $\G$ is a $\Gamma$-graded \'etale groupoid and $\widetilde{F}$ is a $\G$-sheaf such that the grading on $\G$ is compatible with the grading on $\widetilde{F}$, then we call $\widetilde{F}$ a graded $\G$-sheaf of sets.

\begin{definition} \label{G-sheaf of sets}
	Let $\G$ be a $\Gamma$-graded \'etale groupoid, graded by the continuous functor $\kappa:\G~\rightarrow~\Gamma$.
	A $\G$-sheaf $E$ is called a \emph{graded $\G$-sheaf of sets} if there is a continuous map $\kappa:E\to \Gamma$ (named again $\kappa$) such that $\kappa(eg)=\kappa(e)\kappa(g)$, whenever $p(e)=\cod(g)$.
\end{definition}

We write $E=\bigsqcup_{\gamma\in \Gamma} E_\gamma$, where $E_\gamma :=\kappa^{-1}(\gamma)$, and we call $E_\gamma$ the $\gamma$-\emph{component}. A morphism $\phi:E\to F$ of $\G$-sheaves is a \emph{graded morphism} if $\phi(E_\g)\subseteq F_\g$ for any $\g \in \Gamma$. The category of all graded $\G$-sheaves of sets with graded morphisms is denoted by $\mathcal B^{\gr}\G$. There is a forgetful functor $U: \mathcal B^{\gr}\G \rightarrow \mathcal B\G$.

To motivate Definition \ref{G-sheaf of R-mods}, let $\Gamma$ be a group, $X$ a topological space, and $R$ a unital commutative ring. Consider $R$ as a $\Gamma$-graded ring with the trivial grading (i.e., $R = R_\ep$). Let $F$ be a presheaf of graded $R$-modules; i.e., a contravariant functor from $\mathcal{O}_X$ to the category of graded $R$-modules. Applying the sheaf space functor to $F$ one obtains a sheaf of $R$-modules $(\widetilde{F},p_F)$ in which each stalk is a $\Gamma$-graded $R$-module. Suppose, in addition, $\G$ is a $\Gamma$-graded \'etale groupoid and $\widetilde{F}$ is a $\G$-sheaf of $R$-modules such for each $g \in \G_\gamma$ the homomorphism $R_g: E_{\cod(g)} \to E_{\dom(g)}$ implemented by $g$ maps $\left(E_{\cod(g)}\right)_\alpha$ to $\left(E_{\dom(g)}\right)_{\alpha\gamma}$, for all $\alpha \in \Gamma$. We call such an object a graded $\G$-sheaf of $R$-modules.
\begin{definition}\label{G-sheaf of R-mods}
	 Let $\G$ be a $\Gamma$-graded \'etale groupoid and let $E$ be a $\G$-sheaf of $R$-modules. Then $E$ is called a \emph{graded $\G$-sheaf of $R$-modules} if:
\begin{enumerate}[\rm({C}1)]
	\item For any $x\in \G^{(0)}$, $E_x=\bigoplus_{\gamma\in \Gamma} 
	(E_x)_\gamma$, where $(E_x)_\gamma$ are $R$-submodules of $E_x$;	
	\item  $E_\gamma := \bigcup_{x \in \G^{(0)}} (E_x)_\gamma$ is open in $E$ for every $\gamma \in \Gamma$;
	\item $E_\gamma \G_\delta\subseteq E_{\gamma\delta}$ for every $\gamma, \delta \in \Gamma$.
\end{enumerate}
\end{definition}
We call $E_\gamma$ the $\gamma$-\textit{homogeneous component} of $E$, and denote the \textit{homogeneous elements} of $E$ by $E^h:=\bigcup_{\gamma \in \Gamma}E_\gamma$. Note that the degree map $\kappa: E^h \to \Gamma$, $s_\ga \mapsto \gamma$, where $s_\ga\in E_\gamma$, is continuous, and (C3) can be interpreted as $\kappa(eg) = \kappa(e)\kappa(g)$ for every $e \in E^h$ and any $g \in \G$ such that $p(e) = \cod(g)$. A morphism $\phi:E\to F$ of $\G$-sheaves of $R$-modules is a \emph{graded morphism} if $\phi(E_\g)\subseteq F_\g$ for any $\g \in \Gamma$. The category of graded $\G$-sheaves of $R$-modules with graded morphisms will be denoted $\Gr_R\G$.

For a graded $\G$-sheaf of $R$-modules $E$, and $\alpha \in \Gamma$, we define the \textit{$\alpha$-shifted graded $\G$-sheaf of $R$-modules}:
\begin{align}\label{shift-sheaf}
E(\alpha) = \bigcup_{x \in \G^{(0)}} E(\alpha)_x = \bigcup_{x \in \G^{(0)}}\bigoplus_{\gamma \in \Gamma}(E(\alpha)_x)_\gamma, &&\text{where } (E(\alpha)_x)_\gamma = (E_x)_{\alpha \gamma}.
\end{align}
As an ungraded sheaf, $E(\alpha)$ is identical to $E$ but the grading is shifted by $\alpha$ (compare with (\ref{eq:M-shifted})). The $\gamma$-homogeneous component of $E(\alpha)$ is just the $\alpha\gamma$-homogeneous component of $E$. For $\alpha \in \Gamma$, the shift functor $\mathcal{T}_\alpha: \Gr_R\G \to \Gr_R\G$, $E \mapsto E(\alpha)$, is an auto-equivalence.

\begin{lemma} \label{projection-lem}
	If $\G$ is a $\Gamma$-graded \'etale groupoid and $E$ is a graded $\G$-sheaf of $R$-modules, then the natural projection onto the $\gamma$-homogeneous component of $E$ (i.e., $\pi_\gamma: E \to E_\gamma$, $e \mapsto e_\gamma$) is continuous.
\end{lemma}

\begin{proof}
	Let $Y \subseteq E_\gamma$ be open. Let $D = \{\delta_1, \dots, \delta_n\} \subseteq \Gamma \setminus \{\gamma\}$. Then the set \[T(D) = \{e_\gamma + e_{1} + \dots + e_{_n} \mid e_\gamma \in Y, e_{1} \in E_{\delta_1}, \dots, e_{n}\in E_{\delta_n}\}\] is open in $E$ because $n$-fold addition $E \times_{\G^{(0)}} \dots \times_{\G^{(0)}} E \to E$ is an open map, and $Y  \times_{\G^{(0)}} E_{\delta_1} \times_{\G^{(0)}} \dots \times_{\G^{(0)}} E_{\delta_n}$ is an open set in $E \times_{\G^{(0)}} \dots \times_{\G^{(0)}} E$. Thus \[\pi_\gamma^{-1}(Y) = \{e \in  E \mid e_\gamma \in Y \} = \bigcup\big\{T(D) \mid D \text{ is a finite subset of }\Gamma\setminus \{\gamma\}\big\}\]
	is open, proving that $\pi_\gamma$ is continuous.
\end{proof}

Applying Lemma \ref{useful}, we can also conclude that $\pi_\gamma: E \to E_\gamma$ is a local homeomorphism.

\subsection{Dade's Theorem for $\G$-sheaves of $R$-modules}

Let $\G$ be a $\Gamma$-graded ample groupoid, and let $(E,p)$ be a $\G_\ep$-sheaf of $R$-modules. Define the graded $\G$-sheaf $E \otimes_{\G_\ep} \G$ as $\bigcup_{x \in \G^{(0)}} E \otimes \G ^x$ where $E \otimes \G ^x$ is the $R$-module generated by the symbols $\{e \otimes g \mid e \in E, g \in \G ^x, p(e) = \cod(g)\}$ subject to the relations:
\begin{align*}
	e \otimes g + e' \otimes g = (e + e') \otimes g, &&
	(re)\otimes g = r(e \otimes g), && eh \otimes g = e\otimes hg,
\end{align*}
for all $e, e' \in E$, $g \in \G$, $r \in R$, and $h \in \G_\ep$. Indeed, $E \otimes_{\G_\ep} \G$ carries the structure of a $\G$-sheaf of $R$-modules as follows. Define $p(e) = x$ if $e \in E \otimes \G^x$. For $e = \sum e_i \otimes g_i \in E \otimes_{\G_\ep} \G$ and $g \in \G$ with $p(e) = \cod(g)$, define $eg = \sum e_i \otimes g_i g$. The topology on $E \otimes_{\G_\ep} \G$ is generated by the basis of open sets
\begin{align} \label{sfU}
&\Big\{Z(t,U) \mid U \in B^{\rm co}(\G^{(0)}),\ t \in \Gamma_c(E)\otimes_{A_R(\G)_\ep} A_R(\G)\Big\},\text{ where}\\ \notag
&Z(s_1 \otimes f_1 + \dots + s_n \otimes f_n, U) := \bigg\{ \sum_{g \in {\G^x}}\big(f_1(g) s_1(\cod(g)) + \dots + f_n(g) s_n(\cod(g))\big) \otimes g \Bigm| x \in U \bigg\}.
\end{align}
(The origin of this complicated-looking basis becomes apparent in the upcoming lemma.)
Assigning to each $e \otimes g$ the degree $\kappa(e \otimes g) = \kappa(g)$, and noting that the relations defining $E \otimes \G^x$ are homogeneous, we have $E \otimes \G^x = \bigoplus_{\gamma \in \Gamma} E \otimes \G^x_\gamma$. Moreover, $(E\otimes_{\G_\ep}\G)_\gamma$ is open because it is a union of all basic open sets $Z(t,U)$ for $t \in \Gamma_c(E) \otimes A_R(\G)_\gamma$. This makes $E \otimes_{\G_\ep} \G$ a graded $\G$-sheaf of $R$-modules.

\begin{remark}
	In an earlier version of this paper, we gave a different (longer and more difficult) construction of the sheaf $E \otimes_{\G_\ep} \G$. Essentially, $E \otimes \G^x$ is a quotient of the free $R$-module generated by $E \times_{\G^{(0)}} \G^x$ and the basis of open sets in $E \otimes_{\G_\ep} \G = \bigcup_{x \in \G^{(0)}}E \otimes \G^x$ consists of finite sums of subsets $\{e \otimes g \mid (e,g) \in s(W)\times_{\G^{(0)}} B\}$ where $W$ is a compact open subset of $\G^{(0)}$, $s: W \to E$ is a continuous local section, and $B \subseteq \G$ is a compact open bisection.
\end{remark}

Alternatively, to build a $\G$-sheaf of $R$-modules from the $\G_\ep$-sheaf of $R$-modules $(E,p)$, we can apply the following sequence of functors:
\begin{equation}\label{three functors}\xymatrixrowsep{0.65pc}
\xymatrix{ 	\Modd_R\G_\ep
	 \ar[r]^{\Gamma_c}  & \Modd A_R(\G_\ep)
	 \ar[r]^{\mathcal{J}} & \Modd A_R(\G) \ar[r]^{\operatorname{Sh}} & \Modd _R \G \\
	 (E,p) \ar@{|->}[r] & \Gamma_c(E) \ar@{|->}[r]  & \Gamma_c(E) \bigotimes_{A_R(\G)_\ep} A_R(\G) \ar@{|->}[r] & \operatorname{Sh}\left(\Gamma_c(E) \bigotimes_{A_R(\G)_\ep} A_R(\G)\right).
}
\end{equation}

\begin{lemma} \label{iso1}
	Let $(E,p)$ be a $\G_\ep$-sheaf of $R$-modules, where $\G$ is a $\Gamma$-graded ample groupoid. Then $E \bigotimes_{\G_\ep}\G$ and $\operatorname{Sh}\big(
	\Gamma_c(E) \bigotimes_{A_R(\G)_\ep} A_R(\G)
	\big)$ are isomorphic $\G$-sheaves of $R$-modules.
\end{lemma}

\begin{proof}
	Let $M = \Gamma_c(E)$ and $N = M \bigotimes_{A_R(\G)_\ep} A_R(\G)$. By definition, $\operatorname{Sh}(N) = \widetilde{N} =  \bigcup_{x \in \G^{(0)}} {\widetilde{N}}_x$ where \begin{align*}
	\widetilde{N}_x &=  \lim_{\substack{\longrightarrow \\ x \in U }}N \bm{1}_U =  \lim_{\substack{\longrightarrow \\ x \in U }} \Big(M \bigotimes_{A_R(\G)_\ep} A_R(\G)*\bm{1}_U\Big) = M \bigotimes_{A_R(\G)_\ep} \Big(\lim_{\substack{\longrightarrow \\ x \in U }}A_R(\G)*\bm{1}_U \Big) \cong M \bigotimes_{A_R(\G)_\ep} R\G^x.
	\end{align*}
	For a justification of the isomorphism in the last step, see Example \ref{first-ex}. Define
	\begin{align} \label{pix}
		&\pi_x: M \bigotimes_{A_R(\G)_\ep} R\G^x \to E \otimes \G^x, & s \otimes g \mapsto s(\cod(g)) \otimes g, &&
		 \text{for all } s \in M,\ g \in \G^x;
		 \\ \label{sigmax}
		&\sigma_x: E \otimes \G^x \to M \bigotimes_{A_R(\G)_\ep} R\G^x, & \sum_i e_i \otimes g_i \mapsto \sum_i s_i \otimes g_i, && \text{where } s_i \in M, e_i = s_i(\cod(g_i)), g_i \in \G^x.
	\end{align}
	To show that $\sigma_x$ is unambiguously defined, suppose $e \otimes g \in E \otimes \G^x$ and $s, t \in M$ have $s(\cod(g)) = t(\cod(g)) = e$. Pick some neighbourhood $U$ of $\cod(g)$ in $\G^{(0)}$. Since $s(\cod(g)) \in s(U) \cap t(U)$ and $s(U)$ and $t(U)$ are open in $E$, there exists some open set $A \subseteq s(U)\cap t(U)$ with $s(\cod(g)) \in A$, and $s$ and $t$ agree on $W:= p(A)$. Thus, applying (\ref{ARG-mod}) and (\ref{RGx-action}), we have \[s \otimes g = s \otimes \bm{1}_{W} g = s \bm{1}_{W} \otimes g = t \bm{1}_{W} \otimes g = t \otimes \bm{1}_{W}g = t \otimes g.\]
	We omit some details that are not difficult to check: $\sigma_x$ respects the relations on $E \otimes \G ^x$, and $\pi_x \sigma_x$ and $\sigma_x \pi_x$ are the identity homomorphisms. Since $E \otimes_{\G_\ep} \G = \bigcup_{x \in \G^{(0)}} E\otimes \G ^x$ and $\operatorname{Sh}(N) = \bigcup_{x \in \G^{(0)}}\widetilde{N}_x$, the functions (\ref{pix}) and (\ref{sigmax}) extend to functions $\pi: \operatorname{Sh}(N) \to E \otimes_{\G_\ep} \G$ and $\sigma: E \otimes_{\G_\ep} \G \to \operatorname{Sh}(N)$ which are inverse to each other. One can check that $\sigma$ and $\pi$ are $\G$-equivariant, and that the topology transferred from $\operatorname{Sh}(N)$ to $E \otimes_{\G_\ep} \G$ is the one generated by the basis in (\ref{sfU}). Conclude that $\pi$ and $\sigma = \pi^{-1}$ are isomorphisms of $\G$-sheaves of $R$-modules.
\end{proof}

Define the \textit{induction functor}, mapping a $\G_\ep$-sheaf of $R$-modules to a graded $\G$-sheaf of $R$-modules:
\begin{align}\label{funII}
\mathcal J: \Modd_{R}\G_\varepsilon &\longrightarrow \mathcal \Gr_{R}\G,\\*
E&\longmapsto E\otimes_{\G_\varepsilon} \G. \notag
\end{align}
If $E, F$ are $\G_\ep$-sheaves of $R$-modules and $\phi: E \to F$ is a morphism, then we define
\begin{align*}  
\mathcal{J}(\phi) = \phi \otimes \id : E \otimes_{\G_\ep} \G &\longrightarrow F \otimes_{\G_\ep} \G \\
\sum e_i \otimes g_i &\longmapsto \sum \phi(e_i) \otimes g_i. \notag
\end{align*}
One can check that $\mathcal{J}(\phi)$ is continuous, $\G$-equivariant, and restricts to graded homomorphisms on the stalks. That is, $\mathcal{J}(\phi)$ is a morphism of graded $\G$-sheaves of $R$-modules.

On the other hand, if $E$ is a graded $\G$-sheaf of $R$-modules, it is easy to see that $E_\varepsilon:=\bigcup_{x\in \G^{(0)}} (E_x)_\varepsilon$ is a $\G_\varepsilon$-sheaf of $R$-modules. 
Moreover, a graded morphism $\phi:E\rightarrow F$ restricts to a morphism $\phi_\varepsilon:E_\varepsilon\rightarrow F_\varepsilon$. This gives rise to a \textit{restriction functor}, mapping a graded $\G$-sheaf of $R$-modules to a $\G_\ep$-sheaf of $R$-modules:
\begin{align}\label{funI}
\mathcal I: \Gr_{R}\G &\longrightarrow \mathcal \Modd_{R}\G_\varepsilon,\\
E&\longmapsto E_\varepsilon, \notag\\
\phi &\longmapsto \phi_\varepsilon \notag.
\end{align}

Recall that for $E$, a $\G_\ep$-sheaf of $R$-modules, we have $(E \otimes_{\G_\ep} \G)_\ep = \bigcup_{x \in \G^{(0)}} E \otimes \G^x_\ep$. For every $x \in \G^{(0)}$ there is an isomorphism $E \otimes \G^x_\ep \to E_x$, sending $e \otimes g \mapsto eg$, with inverse sending $e \mapsto e \otimes p(e)$. This yields map:
\begin{align} \label{mapI}
	\eta:\mathcal{IJ}(E) = \mathcal{I}(E \otimes_{\G_\ep}\G) = (E \otimes_{\G_\ep}\G)_\ep &\overset{\cong}{\longrightarrow} E \\ \notag
	\eta:\sum e_i \otimes g_i &\longmapsto \sum e_i g_i.
\end{align}
In a moment we shall prove that $\eta$ is an isomorphism of $\G_\ep$-sheaves of $R$-modules and, indeed, $\eta$ is a natural transformation from $\mathcal{IJ}$ to the identity, so we have an isomorphism of functors $\mathcal{IJ} \cong \id_{\Gr_R\G}$.

On the other hand, there is a natural transformation $\mathcal{JI} \to \id_{\Modd_R\G}$. Namely, for a graded $\G$-sheaf $E$, and $x \in \G^{(0)}$, there is a homomorphism $E_\ep \otimes_{\G_\ep} \G^x \to E_x$ sending $e \otimes g \mapsto eg$. This yields a map:
\begin{align} \label{mapII}
	\theta:\mathcal{JI}(E) = \mathcal{J}(E_\ep) = E_\ep \otimes_{\G_\ep} \G & \longrightarrow E\\ \notag
	\theta: \sum e_i \otimes g_i &\longmapsto \sum e_ig_i.
\end{align}

The following lemma formalises an important part of the preceding discussion.

\begin{lemma} \
	\begin{enumerate}[\rm(1)]
	\item If $(E,p)$ is a $\G_\ep$-sheaf of $R$-modules, $\eta$ from (\ref{mapI}) is an isomorphism of $\G_\ep$-sheaves of $R$-modules.
	\item If $(E,p)$ is a graded $\G$-sheaf, $\theta$ from (\ref{mapII}) is a graded morphism of $\G$-sheaves of $R$-modules.
	\end{enumerate}
\end{lemma}

\begin{proof}
	It is clear that $\eta$ is invertible, $\G_\ep$-equivariant, and restricts to $R$-module isomorphisms on the stalks. To show that $\eta$ is continuous, take a basic open set $s(U) \subseteq E$ where $U \in B^{\rm co}(\G^{(0)})$ and $s: U \to E$ is a continuous local section. Then $\eta^{-1}(s(U)) = Z(s \otimes \bm{1}_U, U)$ (see (\ref{sfU})) is open in $(E \otimes_{\G_\ep} \G)_\ep$.

	Clearly, $\theta$ restricts to homomorphisms on the stalks, and it is graded and $\G$-equivariant. Suppose $Z(t,U)$ is a basic open set in $E_\ep \otimes_{\G_\ep}\G$, where $t = \sum s_i \otimes f_i \in \Gamma_c(E_\ep)\bigotimes_{A_R(\G)_\ep} A_R(\G)$, and $U \in B^{\rm co}(\G^{(0)})$. A short calculation reveals that $\theta\left(Z\left(t, U\right)\right) = \big(\sum s_i f_i \big)(U)$, so it is open in $E$ because $\sum s_i f_i \in \Gamma_c(E)$ (see (\ref{ARG-mod})). Thus $\theta$ is an open map, and by Lemma \ref{useful}, $\theta$ is continuous. 
\end{proof}

We are in a position to prove a groupoid version of Dade's Theorem (see Theorem~\ref{dadesthm}, \cite[Theorem 2.8]{dade1}, and \cite[\S1.5]{rh}).

\begin{theorem}\label{strgro}
	Let $\G$ be a $\Gamma$-graded ample groupoid. Then $\G$ is strongly graded if and only if the functors $\mathcal I$ and $\mathcal J$ {\upshape (see (\ref{funII}),(\ref{funI}))} are mutually inverse equivalences of categories.
\end{theorem}

\begin{proof}
	($\Rightarrow$) Recall that $\mathcal{IJ} \cong \id_{\Modd _R\G}$ holds for any $\G$. So, we need to prove that $\G$ is strongly graded implies $\mathcal{JI} \cong \id_{\Gr _R\G_{\ep}}$. For a graded $\G$-sheaf
	$E$,  we have  $\mathcal J  \mathcal I (E) = E_\varepsilon \otimes_{\G_\varepsilon} \G$. Assuming that $\G$ is strongly graded, we show that the natural transformation 
	$\theta:E_\varepsilon \otimes_{\G_\varepsilon} \G \rightarrow E$ (see (\ref{mapII}))
	is an isomorphism of graded
	$\G$-sheaves.  Since $\G$ is
	strongly graded, using Lemma~\ref{stg} (2), it follows that for $\ga, \de \in \Ga$,
	\begin{equation}\label{robin}
	\bigcup_{x\in \G^{(0)}} (E_x)_{\ga \de} =  \bigcup_{x\in \G^{(0)}} (E_x)_{\ga\de} \G_\varepsilon =  \bigcup_{x\in \G^{(0)}} (E_x)_{\ga \de} \G_{\de^{-1}}\G_{\de}  \subseteq 
	\bigcup_{x\in \G^{(0)}} (E_x)_{\ga}  \G_{\de} \subseteq \bigcup_{x\in \G^{(0)}} (E_x)_{\ga \de}.
	\end{equation}
	Thus $E_\gamma \G_\delta= \bigcup_{x\in \G^{(0)}} (E_x)_{\ga} \G_{\de} = \bigcup_{x\in \G^{(0)}} (E_x)_{\ga \de} = E_{\gamma \delta}$. Therefore, 
	$\theta
	( E_\varepsilon \otimes_{\G_\varepsilon} \G_{\ga} ) =  E_\varepsilon \G_{\ga}= E_\gamma$, which shows that $\theta$ is
	surjective.
	
	For the injectivity of $\theta:E_\varepsilon \otimes_{\G_\varepsilon} \G \rightarrow E$, it is enough to show that the restriction 
	$\theta |_x:E_\varepsilon \otimes_{\G_\varepsilon} \G^x \rightarrow E_x$ is injective for any $x\in \G^{(0)}$. In turn, it is enough to show that the restriction $(\theta |_x)_\gamma :E_\varepsilon \otimes_{\G_\varepsilon} \G^x_\gamma \rightarrow (E_x)_\gamma$ is injective for any $\gamma \in \Gamma$. Lemma~\ref{stg} (2) yields $\G_\varepsilon=\G_{\gamma^{-1}}\G_\gamma$, so we can fix an $h\in \G_\gamma$ such that $x=h^{-1}h$. Define an $R$-homomorphism $(\psi |_x)_\gamma: (E_x)_\gamma \rightarrow E_\varepsilon \otimes_{\G_\varepsilon} \G^x_\gamma$ by $e\mapsto eh^{-1} \otimes h$. Now, if $e \in (E_x)_\gamma$, 
	\[(\theta |_x)_\gamma (\psi |_x)_\gamma (e)=(\theta |_x)_\gamma (eh^{-1} \otimes h)=eh^{-1}h=ex=e.\] 
	On the other hand, if $\sum e_i \otimes g_i \in E_\ep \otimes_{\G_\ep}\G^x_\gamma$,
	\begin{multline*}(\psi |_x)_\gamma (\theta |_x)_\gamma \left(\sum e_i\otimes g_i\right)=(\psi |_x)_\gamma \left(\sum e_i g_i\right)\\=\sum e_i g_ih^{-1}\otimes h=\sum e_i\otimes g_ih^{-1} h_i=\sum e_i\otimes g_ix=\sum e_i\otimes g_i.
	\end{multline*}
	The maps $(\psi |_x)_\gamma$ and $(\theta|_x)_\gamma$ are inverses, which implies that $\theta$ is injective. 
	Thus $\mathcal J  \mathcal I (E)= E_\varepsilon \otimes_{\G_\varepsilon} \G\cong E$. Since the morphisms $\eta$ and $\theta$ are natural, this shows that 
	$\mathcal J  \mathcal I \cong \id_{\Gr_R\G}$.

	($\Leftarrow$) Assume $\mathcal I$ and $\mathcal J$ are mutually inverse equivalences (under (\ref{mapI}) and (\ref{mapII})).
	Since $\theta$ is an isomorphism it follows that for any graded $\G$-sheaf of $R$-modules $E$, and any $\gamma\in \Gamma$, $y\in \G^{(0)}$, we have 
	\begin{equation}\label{nohngh}
	E_\varepsilon {\G^y_\gamma} = (E_y)_\gamma.
	\end{equation} We consider the graded $\G$-sheaf of $R$-modules $R\G$ constructed as follows: $R\G:=\bigcup_{x\in \G^{(0)}}R\G^x$, where $R\G^x = \bigoplus_{\gamma \in \Gamma}R\G^x_\gamma$ is a free $R$-module with basis $\G^x$ (see Example \ref{first-ex}). Let $\beta, \gamma \in \Gamma$, and consider the \textit{$\beta$-shifted} graded $\G$-sheaf $E:=R\G(\beta)$ where $(E_x)_\gamma = R\G^x_{\beta\gamma}$; see (\ref{shift-sheaf}). From~(\ref{nohngh}) we get 
	\[
	E_\ep \G_\gamma^y = \bigg(\bigcup_{x\in \G^{(0)}}(E_x)_\varepsilon \bigg) {\G^y_\gamma} =(E_y)_\gamma,
	\]
	which is to say
	\[ \bigg(\bigcup_{x\in \G^{(0)}}R\G^x_\beta \bigg) \G^y_\gamma =R\G^y_{\beta\gamma}\]
	for every $x \in \G^{(0)}$. In particular if $g\in \G_{\beta\gamma}^y$, then the above equation implies that $g\in \G_\beta \G_\gamma$.
	Thus  $\G_{\beta\gamma} \subseteq \G_{\beta} \G_{\gamma}$. This shows that $\G$ is strongly graded. 
\end{proof}

\begin{remark}
	A version of Theorem~\ref{strgro} can also be written for $\G$-sheaves of sets, and then it is only necessary to assume $\G$ is \'etale. That is, the categories $\mathcal B^{\gr}\G$ and $\mathcal B\G_\varepsilon$ are equivalent (under similar functors $\mathcal I$ and $\mathcal J$) if and only if $\G$ is a strongly graded groupoid. An interesting question is whether Theorem \ref{strgro} could be generalised for \'etale groupoids.
\end{remark}

In the following theorem, we find that the property of being strongly graded is transferred from an ample groupoid to its Steinberg algebra.

\begin{theorem}\label{strong}
	Let $\G$ be a $\Gamma$-graded ample groupoid. Then $\G$ is strongly graded if and only if $A_R(\G)$ is strongly graded. 
\end{theorem}

\begin{proof}
	($\Rightarrow$) Assume $\G$ is strongly $\Gamma$-graded. Fix $\gamma, \delta \in \Gamma$, and suppose $U \in B_{\gamma \delta}^{\rm co}(\G)$. We claim it is possible to write $\bm{1}_U = \sum_j r_j (f_j * g_j)$,  where $r_j \in R$, $f_j \in A_R(\G)_{\gamma}$, and $g_j \in A_R(\G)_{\delta}$ for all $j$. Fix some $y \in U$. Since $\G$ is strongly $\Gamma$-graded, there exist morphisms $p \in \G_{\gamma}$ and $q \in \G_{\delta}$ such that $y = pq$. From the continuity of groupoid multiplication, there are compact open bisections $V_y \in B_\gamma^{\rm co}(\G)$, containing $p$, and $W_y \in B_\delta^{\rm co}(\G)$, containing $q$, such that $y = pq \in V_y W_y \subseteq U$.
	Therefore,
	$U = \bigcup_{y \in U} V_y W_y$ and it can reduce to a finite union $U = \bigcup_{i = 1}^N V_i W_i$ because $U$ is compact.
	The principle of inclusion-exclusion yields
	\begin{equation*}
	\bm{1}_U =	\sum_{j = 1}^N (-1)^{j-1} \sum_{\substack{I \subseteq \{1, \dots, N\} \\ |I|=j}}\bm{1}_{\cap_{i \in I}V_i W_i}.
	\end{equation*}
	Examining the terms,
	\begin{align*}
	\bm{1}_{\cap_{i \in I}V_i W_i}
	= \bm{1}_{V_{i_1}W_{i_1} \cap \dots \cap V_{i_j}W_{i_j}}
	&= \bm{1}_{(V_{i_1}W_{i_1}  \cap \dots \cap V_{i_j}W_{i_j})W_{i_1}^{-1}W_{i_1}}\\
	&= \bm{1}_{(V_{i_1}W_{i_1} \cap \dots \cap V_{i_j}W_{i_j})W_{i_1}^{-1}}*\bm{1}_{W_{i_1}} \in A_R(\G)_\gamma \!*\! A_R(\G)_\delta.
	\end{align*}
	The above calculation uses the fact that $W_{i_1}$ is a bisection, which implies $W_{i_1}^{-1} W_{i_1} = \dom(W_{i_1})$. Therefore, $\bm{1}_U \in A_R(\G)_\gamma\!*\! A_R(\G)_\delta$ as claimed. Since the functions $\{\bm{1}_U \mid U \in B_{\gamma \delta}^{\rm co}(\G)\}$ span $A_R(\G)_{\gamma \delta}$, it follows that $A_R(\G)_{\gamma\delta} \subseteq A_R(\G)_\gamma \!* \! A_R(\G)_\delta$ and therefore $A_R(\G)$ is strongly graded.
	
	($\Leftarrow$) Suppose $\G$ is not strongly graded. Then there exists a pair $\gamma, \delta \in \Gamma$ and some $g \in \G_{\gamma \delta}$ such that $g \notin \G_\gamma \G_\delta$.
	Let $W \in B^{\rm co}_{\gamma \delta}(\G)$ be a neighbourhood of $g$, so $\bm{1}_W \in A_R(\G)_{\gamma\delta}$ and $\bm{1}_W(g) = 1$. It is straightforward to check that
	\[
	A_R(\G)_\gamma \!*\! A_R(\G)_\delta = \Span_R\{\bm{1}_{U}\!*\! \bm{1}_{V} \mid U \in B_{\gamma}^{\rm co}(\G),\ V \in B_{\delta}^{\rm co}(\G) \}.
	\]
	If it were true that $\bm{1}_{W} \in A_R(\G)_\gamma \!*\! A_R(\G)_\delta$ then it would be possible to write
	\[
	\bm{1}_{W} = \sum_{j  = 1}^n r_j (\bm{1}_{U_j}\!*\!\bm{1}_{V_j}) = \sum_{j = 1}^n r_j \bm{1}_{U_jV_j}, 
	\]
	where each $U_j \subseteq \G_\gamma$ and $V_j \subseteq \G_\delta$. In particular, this would require for at least one $j$ that $\bm{1}_{U_jV_j}(g) \ne 0$, which would require $g \in U_j V_j \subseteq \G_\gamma \G_\delta$. This is a contradiction. Therefore, $A_R(\G)$ is not strongly $\Gamma$-graded.
\end{proof}

In fact, using a diagram of functors, Theorem \ref{strong} can be derived from Theorem \ref{strgro} and vice versa. To this end, we prove a graded version of Steinberg's Equivalence Theorem from \cite{steinberg2}.

\begin{lemma} \label{works1}
If $\G$ is a $\Gamma$-graded ample groupoid and $M$ is a graded right $A_R(\G)$-module, then $\operatorname{Sh}(M)$ is a graded $\G$-sheaf of $R$-modules. Moreover, $\operatorname{Sh}: \Gr A_R(\G) \to  \Gr_R\G$ is a functor.
\end{lemma}
\begin{proof}
Since $M$ is a graded right $A_R(\G)$-module, it is also a graded $R$-module (where $R = R_\ep$ has the trivial grading) and the presheaf $M$ defined in (\ref{MU}) takes values in the category of graded $R$-modules. Concretely, $M(U)  = \bigoplus_{\gamma \in \Gamma} M(U)_\gamma = \bigoplus_{\gamma \in \Gamma} M_\gamma \bm{1}_U$, for all $U \in B^{\rm co}(\G^{(0)})$. The connecting  homomorphisms $\rho^U_V: M(U) \to M(V)$, $m \mapsto m\bm{1}_V$, are graded homomorphisms. The direct limit $\widetilde{M}_x = \lim_{\substack{\longrightarrow\\x \in U}} M(U)$ is therefore a graded $R$-module, for every $x \in \G^{(0)}$, so $\operatorname{Sh}(M) = \widetilde{M} = \bigcup_{x \in \G^{(0)}}\widetilde{M}_x$ is a sheaf of graded $R$-modules over $\G^{(0)}$. Moreover, for $g \in {^y\G^x_\delta}$ we can pick $B \in B^{\rm co}_\delta(\G)$ containing $g$ and conclude that $[m]_{y} \in (\widetilde{M}_{y})_\gamma$ implies $[m]_{y}g = [m\bm{1}_B]_{x} \in  (\widetilde{M}_{x})_{\gamma \delta}$. This shows $\widetilde{M}_\gamma \G_\delta \subseteq \widetilde{M}_{\gamma \delta}$. Finally, $\widetilde{M}_\gamma = \bigcup_{x \in \G^{(0)}}(\widetilde{M}_x)_\gamma$ is open: for $[m]_x \in (\widetilde{M}_x)_\gamma$ we can assume $m \in M(U)_\gamma$ for some $U$ containing $x$, and thus the open set $(U,m)$ (see (\ref{Um})) has $[m]_x \in (U,m) \subseteq \widetilde{M}_\gamma$.
If $f: M \to N$ is a graded homomorphism of $A_R(\G)$-modules, then clearly $\operatorname{Sh}(f): [m]_x \mapsto [f(m)]_x$ is a graded morphism of $\G$-sheaves. This establishes the functoriality.
\end{proof}

\begin{lemma} \label{works2}
If $\G$ is a $\Gamma$-graded ample groupoid and $(E,p)$ is a graded $\G$-sheaf of $R$-modules, then $\Gamma_c(E)$ is a graded $A_R(\G)$-module. Moreover, $\Gamma_c: \Gr _R\G \to \Gr A_R(\G)$ is a functor.
\end{lemma}
\begin{proof}
Let $N:= \Gamma_c(E) = \{s: \G^{(0)} \to E \mid s \text{ is a compactly supported continuous section}\}$. 
Define $N_\alpha := \{s \in N \mid s(\G^{(0)}) \subseteq E_\alpha\}$ for all $\alpha \in \Gamma$. We first show that $N \subseteq \sum_{\alpha \in \Gamma}N_\alpha$. Define $s_\alpha: \G^{(0)} \to E$ by $s_\alpha(x) = s(x)_\alpha$ for all $x \in \G^{(0)}$. We claim that only finitely many of the $s_\alpha$ are nonzero. Since addition is an open map, for any finite subset $\{\alpha_1, \dots, \alpha_n\} \subseteq \Gamma$, the set $E_{\alpha_1} + \dots + E_{\alpha_n}$ is open in $E$. The collection of all such sets is an open cover of $E$. The nonzero image of $s$, in other words $s(\supp(s))$, is compact because $s$ is continuous and compactly supported. Therefore, reducing to a finite subcover yields a finite subset $\{\alpha_1, \dots, \alpha_n\} \subseteq \Gamma$ such that $s(\supp(s)) \subseteq E_{\alpha_1} + \dots + E_{\alpha_n}$. This proves that only finitely many of the $s_\alpha$ are nonzero.
Next, we claim that each $s_{\alpha_i} \in N_{\alpha_i}$. Indeed $s_{\alpha_i}$ is continuous because it is just $s$ composed with the projection onto the $\alpha_i$-homogeneous component  (see Lemma \ref{projection-lem}). The support of $s_{\alpha_i}$ is compact because it is closed and contained in the support of $s$. This proves that $N \subseteq \sum_{\alpha \in \Gamma} N_\alpha$ and clearly $N_\alpha \cap N_\beta = 0$ if $\alpha \ne \beta$, so $N = \bigoplus_{\alpha \in \Gamma} N_\alpha$. From the definition (see (\ref{ARG-mod})) it is clear $N_\alpha A_R(\G)_\beta \subseteq N_{\alpha \beta}$ for all $\alpha, \beta \in \Gamma$. To finish, note that when $\phi: E \to F$ is a graded morphism of $\G$-sheaves of $R$-modules, $\Gamma_c(\phi): s \mapsto \phi \circ s$ is a graded homomorphism.
\end{proof}

\begin{proposition}
	Let $\G$ be a $\Gamma$-graded ample groupoid, and $R$ a commutative unital ring. Then $\Gamma_c: \Gr_R \G \rightarrow \Gr A_R(\G)$ and $\operatorname{Sh}:\Gr A_R(\G)\rightarrow \Gr_R \G$
	are mutually inverse equivalences of categories, and the equivalences commute with the shift functors $\mathcal{T}_\alpha: \Gr A_R(\G) \to \Gr A_R(\G)$ and $\mathcal{T}_\alpha: \Gr _R\G \to \Gr _R\G$.
\end{proposition}

\begin{proof}
	With the help of Lemmas \ref{works1} and \ref{works2}, it follows from \cite[Theorem 3.5]{steinberg2} that $\Gamma_c \circ \operatorname{Sh} \cong \id_{\Gr A_R(\G)}$ and $\operatorname{Sh} \circ \Gamma_c \cong \id_{\Gr _R(\G)}$, since the natural transformations involved in that theorem are indeed graded isomorphisms. If $E$ is a graded $\G$-sheaf and $\alpha \in \Gamma$, it is immediate that $\Gamma_c(E(\alpha)) = \Gamma_c(E)(\alpha)$. Similarly, if $M$ is a graded $A_R(\G)$-module then $\operatorname{Sh}(M(\alpha)) = \operatorname{Sh}(M)(\alpha)$.
\end{proof}

From this result, we can prove Theorem \ref{strgro} using Theorem \ref{strong} or the other way around (so the theorems are equivalent, and only one of the proofs is really necessary). The key lies in the following commutative diagram: 
\begin{equation} \label{comm diag}
\xymatrix{ 
	\Gr_{R}\G\ar[r]^{\mathcal I}\ar@<-0.5ex>[d]_{\Gamma_c}&\Modd_{R}\G_\varepsilon \ar@<-0.5ex>[d]_{\Gamma_c} \ar[r]^{\mathcal J} & \Gr_{R}\G \ar@<-0.5ex>[d]_{\Gamma_c}\\
	\Gr A_{R}(\G)\ar[r]^{\mathcal I} \ar@<-0.5ex>[u]_{\operatorname{Sh}}&\Modd A_{R}(\G)_\ep \ar@<-0.5ex>[u]_{\operatorname{Sh}} \ar[r]^{\mathcal J}& \Gr A_{R}(\G) \ar@<-0.5ex>[u]_{\operatorname{Sh}}.}
\end{equation}
Lemma \ref{iso1} proves that the square on the right commutes (since the isomorphism from $\mathcal{J}$ to $\operatorname{Sh} \circ \mathcal{J} \circ \Gamma_c$ is natural and graded). It is easy to prove that the square on the left commutes.

\begin{proof}[Proof of Theorem \ref{strong} using Theorem \ref{strgro}]
	If $\G$ is a strongly graded groupoid, by Theorem~\ref{strgro}, $\mathcal I \mathcal J\cong \id_{\Modd_R\G_\varepsilon}$ and $\mathcal J \mathcal I \cong \id_{\Gr_R\G}$ (on the top row). Since the vertical arrows are equivalences, this implies that $\mathcal I \mathcal J\cong \id_{\Modd A_{R}(\G_\varepsilon)}$ and $\mathcal J \mathcal I \cong \id_{\Gr A_{R}(\G)}$  (on the bottom row). Now by Theorem~\ref{dadesthm}, $A_R(\G)$ is a strongly graded ring. The converse is similar.
\end{proof}

\begin{proof}[Proof of Theorem \ref{strgro} using Theorem \ref{strong}]
	If $\G$ is a strongly graded groupoid, then by Theorem \ref{strong}, $A_R(\G)$ is a strongly graded algebra. By Theorem \ref{dadesthm}, $\mathcal{IJ} \cong \id_{\Modd A_R(\G)_\ep}$ and $\mathcal{JI} \cong \id_{\Gr A_R(\G)}$. Using the diagram (\ref{comm diag}), this implies $\mathcal I \mathcal J\cong \id_{\Modd_R\G_\varepsilon}$ and $\mathcal J \mathcal I \cong \id_{\Gr_R\G}$ on the top row. Conversely, if $\mathcal I \mathcal J\cong \id_{\Modd_R\G_\varepsilon}$ and $\mathcal J \mathcal I \cong \id_{\Gr_R\G}$, then $\mathcal{IJ} \cong \id_{\Modd A_R(\G)_\ep}$ and $\mathcal{JI} \cong \id_{\Gr A_R(\G)}$, and Theorem \ref{dadesthm} implies $A_R(\G)$ is strongly graded. Thus $\G$ is strongly graded, by Theorem \ref{strong}.
\end{proof}

\begin{corollary}
	Let $\G$ be a $\Gamma$-graded ample groupoid. The following are equivalent:
	\begin{enumerate}[\rm (1)]
		\item $\G$ is strongly graded;
		\item $A_R(\G)$ is strongly graded;
		\item All the arrows in (\ref{comm diag}) are equivalences.
	\end{enumerate}
\end{corollary}

\section{Applications} \label{applications}

In this section, we use the previous results (in fact, we only use Lemma \ref{stg} and Theorem \ref{strong}) to characterise strongly $\ZZ$-graded and $\ZZ/n\ZZ$-graded Leavitt path algebras, and strongly $\ZZ^k$-graded Kumjian-Pask algebras. At the end, we briefly discuss strongly graded transformation groupoids.

\subsection{Leavitt path algebras} \label{LPA} Leavitt path algebras are $\ZZ$-graded $R$-algebras presented by generators and relations that are determined by a directed graph. For every graph $E$, there is a $\ZZ$-graded ample groupoid $\G_E$ such that $A_R(\G_E)$ is graded isomorphic to the Leavitt path algebra of $E$ with coefficients in $R$ (see \cite[Example 3.2]{CS}). 

\subsubsection{Preliminaries}
We refer to \cite[Definitions 1.2.2]{LPAbook} for the standard terminology used to describe a \textit{directed graph} $E = (E^0, E^1, r, s)$. A vertex is
a \textit{sink} if it emits no edges. A vertex is an \textit{infinite emitter} if it emits infinitely many edges, and graph is \textit{row-finite} if it has no infinite emitters. Sinks and infinite emitters are collectively called \textit{singular vertices}. A vertex that neither receives nor emits any edges is called \textit{isolated}.
In this section, we adopt the convention from \cite{LPAbook} that a path is a sequence of edges $\alpha = \alpha_1 \dots \alpha_n$ where the range of $\alpha_i$ coincides with the source of $\alpha_{i+1}$, i.e., $r(\alpha_i) = s(\alpha_{i+1})$, for all $1 \le i \le n-1$. We write $|\alpha| = n \in \NN\cup \{\infty\}$ for the \textit{length} of the path. We use the notation $E^\star$ for the set of finite paths, and $E^\infty$ for the set of infinite paths. We define the set of \textit{boundary paths} as $\partial E:= E^\infty \cup \{\alpha \in E^\star \mid r(\alpha) \text{ is singular}\}$. If $\alpha \in E^\star$, $p \in E^\star \cup E^\infty$, and  $p = \alpha q$ for some $q \in E^\star \cup E^\infty$, then we say that $\alpha$ is an \textit{initial subpath} of $p$. In this section, we make no restrictions on the cardinality of $E^0$ or $E^1$.

Let $E$ be a graph and let $A$ be a ring. A subset $\{v, e, e^* \mid v \in E^0, e \in E^1 \} \subseteq A$ is called a \textit{Leavitt $E$-family} if the elements of 
$\{v\mid  v \in E^0\}$ are pairwise orthogonal idempotents and the following conditions are satisfied:
\begin{enumerate}[(E1)]
	\item $s(e) e = e r(e) = e$ for all $e \in E^1$,
	
	\smallskip

	\item $e^*s(e) = r(e)e^* = e^*$ for all $e \in E^1$,
\end{enumerate}
\begin{enumerate}[(CK1)]
	\item $e^* f = \delta_{e,f} r(e)$ for all $e, f \in E^1$, and
	
	\smallskip

	\item $v = \sum_{e \in s^{-1}(v)} e e^*$ for every regular vertex $v \in E^0$.
\end{enumerate}

As usual, let $R$ be a commutative ring with unit. The \textit{Leavitt path algebra with coefficients in $R$}, 
which we denote by $L_R(E)$, is the universal $R$-algebra generated by a Leavitt
$E$-family.
Leavitt path algebras have a canonical $\mathbb{Z}$-graded structure with homogeneous components
\[L_R(E)_n = \{\alpha \beta^* \mid  r(\alpha) = r(\beta),\ n = |\alpha|-|\beta| \}.\]

We now describe the boundary path groupoid, which was introduced in \cite{KPRR}. Let $E$ be a graph, and define the \textit{one-sided shift map} $\sigma: \partial E \setminus E^0 \to \partial E$ as follows: 
\[
\sigma(p) := 
\begin{cases}
r(p) &\text{ if } p \in E^\star \cap \partial E \text{ and } |p|=1 \\
p_2 \dots p_{|p|} &\text{ if } p \in E^\star \cap \partial E \text{ and } |p| \ge 2\\
p_2 p_3 \dots &\text{ if } p \in E^\infty
\end{cases}
\]
The $n$-fold composition $\sigma^n$ is defined on 
paths of length $\ge n$ and we understand that $\sigma^0: \partial E \to \partial E$ is the identity map.
The \textit{boundary path groupoid} is the groupoid
\begin{align*}
\G_E :&= \left\{(x, k, y) \in \partial E \times \mathbb{Z} \times \partial E \mid  \sigma^n(x) = \sigma^{n - k}(y) \text{ for some } n \ge \max\{0,k\}\right\}\\\
&= \left\{(\alpha x, |\alpha| - |\beta|, \beta x) \mid  \alpha, \beta \in E^\star,\ x \in \partial E,\ r(\alpha)= r(\beta) = s(x) \right\}, 
\end{align*}
with domain, codomain, multiplication, and inversion maps:
\begin{align*}
\dom(x,k,y) = y, && \cod(x,k,y) = x, &&
(x, k , y)(y, l,z)
= (x, k + l, z), && (x, k, y)^{-1}=(y, -k, x).
\end{align*}
The unit space is
$\G_E^{(0)} = \{(x, 0,x) \mid  x \in \partial E \}$, which we identify with $\partial E$. 
The groupoid $\G_E$ comes with a canonical $\mathbb{Z}$-grading given by the 
functor $\varphi:(p,k,q) \mapsto k$. Since we do not need to work with the topology on $\G_E$, it suffices to say that there is a topology with respect to which $\G_E$ is a $\ZZ$-graded ample groupoid. There is a graded isomorphism $\pi_E: L_R(E)\to A_R(\G_E)$, but we do not need to use it explicitly, so we refer the reader to \cite[Example 3.2]{CS}.

\subsubsection{Strongly graded Leavitt path algebras}

\begin{definition} \label{Y}
	A graph $E$ satisfies \textit{Condition (Y)} if for every $k \in \mathbb{N}$ 
	and every infinite path $p$, there exists an initial subpath $\alpha$ of $p$ and a finite path $\beta$ such that $r(\beta) = r(\alpha)$ and $|\beta| - |\alpha| = k$.
\end{definition}

If $E$ is a graph such that every infinite path contains a vertex that is the base of a cycle, then $E$ satisfies Condition (Y). On the other hand, there exist infinite acyclic graphs that satisfy Condition (Y).

\begin{theorem}
\label{thm:LPA}
	Let $E$ be a graph, and $R$ a unital commutative ring. 
	The Leavitt path algebra $L_R(E)$ is strongly $\mathbb{Z}$-graded if 
	and only if $E$ is row-finite, has no sinks, and satisfies Condition (Y).
\end{theorem}

\begin{proof}
	We prove the statement for $A_R(\G_E)$, since it is graded isomorphic to $L_R(E)$.
	
		($\Rightarrow$) Firstly, suppose $E$ has a singular vertex. 
		Then there is a finite path $\mu \in \partial E$. 
		The morphism $(\mu, 0, \mu) \in \G_E^{(0)}$ cannot be 
		factored in the form $(\mu,|\mu|+1,x)(x,-(|\mu|+1), \mu)$, where 
		$x \in \partial E$, so $(\mu, 0, \mu) \notin {(\G_E)}_{|\mu| + 1} {(\G_E)}_{-(|\mu|+1)}$. 
		Therefore, $\G_E$ is not strongly $\mathbb{Z}$-graded, so neither is $A_R(\G_E)$, 
		according to Theorem \ref{strong}. Secondly, suppose $E$ has no singular vertices, 
		but fails to satisfy Condition (\hyperref[Y]{Y}). This means there is some $k \in \mathbb{N}$, 
		and some infinite path  $p \in E^\infty$, such that for every initial 
		subpath $\alpha$ of $p$, there does not exist a finite path $\beta \in E^\star$ 
		having $r(\beta) = r(\alpha)$ and $|\beta|-|\alpha| = k$. 
		Therefore, the morphism $(p, 0, p) \in \G_E^{(0)}$ does not admit a factoring of the form $(p, 0, p) = (\alpha p', -k, \beta p')(\beta p', k, \alpha p')$. This implies $(p, 0, p) \notin (\G_E)_{-k} (\G_E)_{k}$, so $\G_E$ is not strongly graded, and consequently $A_R(\G_E)$ is not strongly graded.
		
		($\Leftarrow$) Suppose $E$ is row-finite, has no sinks, and satisfies Condition (\hyperref[Y]{Y}). Let $p \in \partial E$ be arbitrary. 
		There are no singular vertices in $E$, so $p$ is an infinite path. For $n\ge0$, we have $(p, n, \sigma^n(p)) \in (\G_E)_n$. For $n < 0$, Condition (\hyperref[Y]{Y}) implies that there exists some initial subpath $\alpha$ of $p$, and a finite path $\beta \in E^\star$ with $r(\beta) = r(\alpha)$ and $|\beta|- |\alpha| = -n$. Then $(p,n,\beta \sigma^{|\alpha|}(p))  \in (\G_E)_n$. Therefore, $p \in \cod((\G_E)_n)$ for every ${n} \in \mathbb{Z}$. By Lemma \ref{stg}~(\ref{it4:stg}), $\G_E$ is strongly graded. By Theorem \ref{strong}, $A_R(\G_E)$ is strongly $\mathbb{Z}$-graded.
\end{proof}

It is also possible to equip $L_R(E)$ with a non-canonical graded structure. One way to do this is to take the quotient grading by a subgroup $n\mathbb{Z} \lhd \mathbb{Z}$. To simplify notation for cosets, let $[k] := k + n\ZZ$.

\begin{proposition}
	Let $E$ be a graph. Then $L_R(E)$ is strongly $\ZZ/n\ZZ$-graded if and only if every singular vertex receives a path of length $n-1$.
\end{proposition}

\begin{proof}
	($\Leftarrow$) Let $\G = \G_E$ and take $x \in \partial E$. If $x$ is infinite, or if $|x| \ge n-1$, then for $0 \le k \le n-1$, we have $(x, k, \sigma^k(x)) \in \G_{[k]}$ so $x \in \cod({\G}_{[k]})$ for any $[k] \in \mathbb{Z}/n\mathbb{Z}$. Otherwise $0 \le |x| < n-1$ and $r(x)$ is a singular vertex. By assumption, there exists $\mu \in E^\star$ of length $n-1$, such that $r(\mu) = r(x)$. For all $0 \le k \le n-1$ we have
	\begin{equation*}
	(x,0,x)
	= \big(x, |x|-k, \sigma^{n-1-k}(\mu)\big) \big(\sigma^{n-1-k}(\mu), k-|x|, x\big),
	\end{equation*}
	so $x \in \cod(\G_{[|x|-k]})$. Therefore $x \in \cod(\G_{[k]})$ for every $[k] \in \mathbb{Z}/n\mathbb{Z}$, so $\G_E$ is strongly $\ZZ/n\ZZ$-graded, by Lemma \ref{stg}~(\ref{it4:stg}). Conclude that $A_R(\G_E) \cong L_R(E)$ is strongly $\mathbb{Z}/n\mathbb{Z}$-graded.
	
	($\Rightarrow$) If $v \in E^0$ is a singular vertex that does not receive a path of length $n-1$, then $v \in \partial E$ but $v \notin \dom(\G_{[n-1]})$. By Lemma~\ref{stg}~(\ref{it3:stg}), $\G$ is not strongly $\mathbb{Z}/n\mathbb{Z}$-graded, so $L_R(E)$ is not strongly $\ZZ/n\ZZ$-graded.
\end{proof}

We now can easily recover one of the main theorems of \cite{rhs}, namely~\cite[Theorem 3.15]{rhs}, and identify a large collection of strongly graded algebras. 

\begin{corollary} \label{recovery}
	Let $E$ be a row-finite graph. 
	
	\begin{enumerate}[\upshape(1)]
		\item  $L_R(E)$ is a strongly $\mathbb Z/2\mathbb Z$-graded ring if and only if $E$ has no isolated vertex.
		\item  $L_R(E)$ is a strongly $\mathbb Z/n\mathbb Z$-graded ring if $E$ has no sink.
	\end{enumerate}
	If $E^0$ is finite:
	\begin{enumerate}[\rm (1)] \setcounter{enumi}{2}
		\item \label{recovery 3} $L_R(E)$ is a strongly $\mathbb Z$-graded ring if and only if $E$ has no sink. 	
	\end{enumerate}
\end{corollary}

\subsection{Kumjian-Pask algebras of higher-rank graphs}

\label{KPA}

\subsubsection{Preliminaries on $k$-graphs}
We view the additive semigroup $\NN^k$ as a category with one object $0 := (0,\dots, 0)$ and equip it with the coordinate-wise partial order
\[
m \le n \iff m_i \le n_i \text{ for } 1 \le i \le k.
\]
With this partial order, $\NN^k$ is a lattice, and we use the notation 
$\vee$ for the supremum (coordinatewise maximum) and $\wedge$ for the infimum 
(coordinatewise minimum). We denote the usual generators in $\NN^k$ by $\{e_i \mid  1 \le i \le k\}$. 
The partial order $\le$ also extends to the abelian group $\ZZ^k$.
\begin{definition}
A \textit{higher-rank graph} of \textit{rank $k$}, or 
\textit{$k$-graph} for short, is a countable small category $\Lambda = (\Lambda^0, \Lambda, r, s)$, 
together with a functor $d: \Lambda \to \NN^k$, called the \textit{degree map}, satisfying the \textit{factorisation property}: for every $\lambda \in \Lambda$ and $m,n \in \NN^k$ with $d(\lambda) = m+n$, there are unique morphisms $\mu, \nu \in \Lambda$ with $d(\mu) = m$, $d(\nu) = n$, and $\lambda = \mu\nu$.
\end{definition} 

We write $\lambda = \lambda(0,m)\lambda(m, m+n)$, where $\lambda(0,m)$ and $\lambda(m,m+n)$ are the unique composable factors of $\lambda$ with degrees $m$ and $n$ respectively. The set $\Lambda^0$ is the set of objects, which we may think of as \textit{vertices}, and we identify each object $v \in \Lambda^0$ with the identity morphism at $v$, which according to the factorisation property is the only morphism $v \to v$ in $\Lambda$. The maps $r, s: \Lambda \to \Lambda^0$ are the \textit{range} and \textit{source} maps respectively. For $n \in \NN^k$, we define $\Lambda^n := d^{-1}(n)$, and call the elements of $\Lambda^n$ \textit{paths of degree $n$}. Note that $1$-graphs are the same as ordinary directed graphs, although the notation differs and the roles of $r$ and $s$ are swapped (see \cite[Example 1.3]{kp}). For a vertex $v \in \Lambda^0$, we use the notation, 
\begin{align*}
v\Lambda &:= \{\lambda \in \Lambda \mid  r(\lambda) = v\},\\
v\Lambda^n &:= \{\lambda \in \Lambda^n  \mid r(\lambda) = v \}.
\end{align*}
We say that $\Lambda$ is \textit{row-finite} if $v\Lambda^n$ is finite for every $v \in \Lambda^0$ and $n \in \NN^k$. A \textit{source} is a vertex $v \in \Lambda^0$ such that $v\Lambda^{e_i} = \emptyset$ for some $1 \le i \le k$.
The $k$-graph $\Lambda$ has no sources if and only if $v\Lambda^n$ is non-empty for every $v \in \Lambda^0$ and $n \in \NN^k$.

We follow \cite{CP} in saying that $\tau$ is a \textit{minimal common extension} of $\lambda, \mu \in \Lambda$ if $d(\tau) = d(\lambda) \vee d(\mu)$, $\tau(0,d(\lambda)) = \lambda$, and $\tau(0, d(\mu)) = \mu$. Not every pair of paths has a common extension and, if $k \ge 2$, a pair of paths may have more than one minimal common extension. Let $\MCE(\lambda, \mu)$ be the set of all minimal common extensions of $\lambda$ and $\mu$, and
\begin{align*}
\Lambda^{\min}(\lambda, \mu) :&= \{(\rho, \tau) \in \Lambda \times \Lambda \mid  \lambda \rho = \mu \tau \in \MCE(\lambda, \mu)\} \\
&= \{(\rho, \tau) \in \Lambda \times \Lambda \mid  \lambda \rho = \mu \tau \text{ and } d(\lambda\rho) = d(\lambda)\vee d(\mu)\}.
\end{align*}
We say that the $k$-graph $\Lambda$ is \textit{finitely aligned} if $\Lambda^{\min}(\lambda, \mu)$ is finite for every $\lambda, \mu \in \Lambda$ or, equivalently, if $\MCE(\lambda, \mu)$ is finite for every $\lambda, \mu \in \Lambda$. Every row-finite $k$-graph is finitely aligned, but there exist finitely aligned $k$-graphs which are not row-finite \cite{CP}. From here on, we assume $\Lambda$ is a finitely aligned $k$-graph.

A subset $E \subseteq v\Lambda$ is said to be \textit{exhaustive} if for every $\lambda \in v\Lambda$, there exists $\mu \in E$ such that $\Lambda^{\min}(\lambda, \mu) \ne~\emptyset$. As in \cite{CP}, we define 
\begin{align*}
v\FE(\Lambda) &:= \{E \subseteq v\Lambda\setminus \{v\} \mid  E \text{ is finite and exhaustive} \}, \\
\FE(\lambda) &:= \bigcup_{v \in \Lambda^0} v\FE(\Lambda).
\end{align*}
If $v \in \Lambda^0$, $\lambda \in v\Lambda$, and $E \in v\FE(\Lambda)$, the set
\[
\Ext(\lambda, E): = \bigcup_{\mu \in E} \left\{ \rho \in \Lambda \mid  (\rho, \tau) \in \Lambda^{\min}(\lambda, \mu) \text{ for some } \tau \in \Lambda\right\}
\]
is a finite exhaustive subset of $s(\lambda)\Lambda$, according to \cite[Lemma C.5]{RSY04}, and the set
\begin{equation} \label{setI}
I(E) := \bigcup_{i = 1}^k \{\lambda(0, e_i)\mid  \lambda \in E, d(\lambda)_i > 0 \}
\end{equation}
is a finite exhaustive subset of $v\Lambda$ \cite[Lemma C.6]{RSY04} whose elements can be viewed as edges (of various colours).

We refer to \cite[Example 3.2]{FMY05} for the definition of the row-finite $k$-graphs $\Omega_{k,m}$, for $m \in (\NN\cup\{\infty\})^k$. Briefly:
\begin{align*} \Omega_{k,m}^0 :&= \{p \in \NN^k \mid  p \le m\}, & \Omega_{k,m} :&= \{(p,q) \in \Omega_{k,m}^0 \times \Omega_{k,m} \mid  p \le q\},\\
r(p,q) &= p, & s(p,q) &= q, & & d(p,q) = q-p.
\end{align*}

\begin{definition}\cite[Definitions 5.1 and 5.10]{FMY05}
Let $\Lambda$ be a $k$-graph. A \textit{boundary path} in $\Lambda$ is a degree-preserving functor $x: \Omega_{k,m} \to \Lambda$ with the property: for all $n \in \NN^k$ with $n \le m$, and all $E \in x(n)\FE(\Lambda)$, there exists some $\lambda \in E$ such that $x(n, n+d(\lambda)) = \lambda$. We write $\partial \Lambda$ for the set of all boundary paths in $\Lambda$, and say that the \textit{degree} of $x \in \partial \Lambda$ is $d(x) := m$.
\end{definition}

We also have the notion of an \textit{infinite path} \cite{ACHR}, which is a degree-preserving functor $x : \Omega_{k,(\infty,\dots,\infty)} \to \Lambda$. In \cite[Proposition 2.12]{W11}, we see that every infinite path is a boundary path. We denote the set of infinite paths by $\Lambda^\infty$, and write $r(x) :=x(0)$ for the \textit{range} of an infinite path. For $v \in \Lambda^0$, we also write
\[
v \Lambda^\infty := \{x \in \Lambda^\infty \mid  r(x) = v \}.
\]
We recall some facts about infinite paths from \cite[Lemma 2.5]{ACHR}. 
\begin{itemize}
	\item Finite paths and infinite paths can be composed: if $\lambda \in \Lambda$ and $x \in s(\lambda)\Lambda^\infty$ then there is a unique $y \in \Lambda^\infty$ such that $y(0,n) = \lambda x(0, n - d(\lambda))$ for all $n \ge d(\lambda)$. In this case we write $\lambda x:= y$.
	
	\smallskip

	\item Infinite paths have a factorisation property: if $x \in \Lambda^\infty$ then there exist unique $x(0,n) \in \Lambda^n$ and $x(n, \infty) \in \Lambda^\infty$ such that $x = x(0, n)x(n, \infty)$. Moreover, $s(x(0,n)) = x(n) = r(x(n,\infty))$.
\end{itemize}

If $\Lambda$ is row-finite and has no sources, then the boundary path space is considerably easier to work with, 
because $\partial \Lambda = \Lambda^\infty$ (see \cite[Proposition~2.12]{W11}). 

\subsubsection{Strongly graded Kumjian-Pask algebras}

\begin{lemma}
\label{lem:hardone}
Let $\Lambda$ be a finitely aligned $k$-graph with no sources.  Suppose $\Lambda$ is not row-finite.  
Then there exists $x \in \partial \Lambda$ and
 $i \in \{1, ..., k\}$ such that $d(x)_i < \infty$.
\end{lemma}

\begin{proof} {Let $k \ge 2$, since the statement for 1-graphs is trivial.} We consider two cases. 
First, suppose there exists $v \in \Lambda^0$ such that $v\operatorname{FE}(\Lambda) = \emptyset$.  
Then $v \in \partial \Lambda$ and we are done.
For the second case, suppose that $v\operatorname{FE}(\Lambda)\neq \emptyset$ for every $v \in \Lambda^0$. Since $\Lambda$ is not row-finite, there exists $v\in \Lambda^0$ and
$i \in \{1, ..., k\}$ such that $|v\Lambda^{e_i}|=\infty$.
  
Since there exists $E \in v\operatorname{FE}(\Lambda)$, there also exists a finite exhaustive set $I(E) \in v\FE(\Lambda)$,
consisting entirely of edges (see (\ref{setI})). So we can find an edge $\lambda_1 \in I(E)$ such that $\Lambda^{\min}(\mu, \lambda_1) \ne \emptyset$ for infinitely many distinct $\mu \in v\Lambda^{e_i}$. It follows that $d(\lambda_1)_i=0$ and  
$|s(\lambda_1)\Lambda^{e_i}|=\infty$.  Similarly, for $n >1$ we can find $\lambda_n \in s(\lambda_{n-1})\Lambda$ such that
\begin{equation}\label{desired property} d(\lambda_n)_i=0 \text{ and } |s(\lambda_n)\Lambda^{e_i}|=\infty.
\end{equation}  
As $n$ increases,  we pick $\lambda_n$ so that the degree of   
$x:= \lambda_1...\lambda_n ... $
is as large as possible in each component.  We do this
by cycling through each of the components $j \in \{1, ..., k\}$ as $n$ increases.
For example, when $n=1$, we start with $j=1$ and look for an
edge $\lambda_1$ with the desired property (\ref{desired property}) and with degree $e_1$.  
If there is no such edge (for example, if $i=1$), we try $j=2$ and look for an edge of degree $e_2$ so on.  
Then for $n=2$, we start with $j=j_1+1$, where  $j_1$ is such that $d(\lambda_1)_{j_1}=1$.
Note that,
\begin{enumerate}
           \item\label{it1:proof} $d(x)_i=0$; and
           \smallskip

           \item\label{it2:proof} If $d(x)_j < \infty$ for some $j$ then 
           there exists $N$ such that $d(\lambda_N)=e_j$ and 
           \[d(\lambda_1...\lambda_N)_{j} = d(x)_j.\]
          
          \smallskip

          \item\label{it3:proof} We claim that for every $n > N$ we have 
           $|x(n)\Lambda^{e_j}|=\infty$ (where $N$ and $j$ are from \eqref{it2:proof}). 
           For otherwise, there exists some $n>N$ such that $E_n:=x(n)\Lambda^{e_j} \in x(n)\operatorname{FE}(\Lambda)$.  
           But then  for every $m>0$ we have a finite exhaustive  set
           \[\operatorname{Ext}(\lambda_{n+m}, E_n) \subseteq s(\lambda_{n+m})\Lambda^{e_j}.\]  
           Since our choice of $x$ makes the degree in each component as large as possible, we eventually will choose 
           a $\lambda_m$ with degree $e_j$, for example in some $I(\Ext(\lambda_{n+m}, E_n))$, which is a contradiction. This verifies the claim.
\end{enumerate}

We show $x \in \partial \Lambda$ using an approach similar to
\cite[Proposition~2.12]{W11}. 
Fix $n \leq d(x)$ and $E \in x(n)\operatorname{FE}(\Lambda)$.  
Define $t \in  \NN^k$ such that for $j \in \{1, ..., k\}$, 
\[t_j := \begin{cases} d(x)_j & \text{if } d(x)_j < \infty \\
          \underset{\lambda \in E}{\operatorname{max}}\{d(\lambda)_j\}&\text{if } d(x)_j = \infty.
         \end{cases}\]
Then $n+ t < d(x)$ and we can consider $x(n, n+t) \in x(n)\Lambda$.
Define 
\[F_t:= \{j \in \{1, ..., k\}\mid x(n+t)\Lambda^{e_j} \text{ is infinite}\}.\]
Then, since $\Lambda$ has no sources,  there exist an infinite number of paths \[\delta \in x(n+t)\Lambda^{\sum_{j\in F_t} e_j},\]
which we label $\delta_1, \delta_2, \dots$.  Since $E$ is finite exhaustive, there exists a single
$\nu \in E$ which has common extensions with infinitely many $x(n, n+t)\delta_q$. For each $q \in \NN$ (by passing to a subsequence and relabelling), there exists $(\alpha_q,\beta_q) \in \Lambda^{\min}(x(n, n+t)\delta_q, \nu)$ whereby,
\begin{equation}\label{eq2}x(n, n+t)\delta_q \alpha_q = \nu \beta_q.\end{equation}
We claim that 
$d(\nu) \leq t$. If $j$ is such that $d(x)_j=\infty$, then $d(\nu)_j < t_j$ by construction. 
Otherwise, suppose $j$ is such that $t_j:=d(x)_j < \infty$.  Then we have $j \in F_t$ by item~\eqref{it3:proof} above.
By way of contradiction, suppose
$d(\nu)_j > t_j$.  Write $\nu = b\nu'$, where $d(b) = t_j+e_j$.
So for each $q \in \NN$ we have from \eqref{eq2}
\[\left(x(n,n+t)\delta_q\alpha_q\right)(0, t+e_j) = x(n, n+t)(\delta_q(0,e_j)) = b\nu'\beta_q(0,t+e_j).\]
But then $|\Lambda^{\operatorname{min}}(x(n,n+t), b)|=\infty$, which contradicts that $\Lambda$ is finitely aligned. Thus $d(\nu)_i\leq d_j$.  Now
\eqref{eq2} gives
$x(n+d(\nu)) = \nu$,
so $x \in \partial \Lambda$.  
\end{proof}

\begin{lemma}
 \label{lem:sources} 
 Let $\Lambda$ be a finitely aligned $k$-graph and suppose that $v \in \Lambda^0$ is a source.
 Then there exists $x \in \partial \Lambda$ and
 $i \in \{1, ..., k\}$ such that $d(x)_i < \infty$.
\end{lemma}

\begin{proof}
By the definition of a source, there exists 
$i \in \{1, ..., k\}$ such that $v\Lambda^{e_i} = \emptyset$.  
By \cite[Lemma~5.15]{FMY05} $v\partial \Lambda \neq \emptyset$ 
so we can chose $x \in v\partial \Lambda$. Then $d(x)_i =0$.  
\end{proof}

\bigskip

We now introduce the boundary path groupoid $G_\Lambda$ 
which was first studied by Yeend in \cite{Yeend}.  The boundary path groupoid 
 is a Hausdorff ample groupoid, with unit space $\partial \Lambda$, that generalises the boundary path groupoid of an ordinary directed graph (or $1$-graph).
We shall use the notation 
\[
\Lambda *_s \Lambda := \{(\lambda, \mu) \in \Lambda \times \Lambda \mid  s(\lambda) = s(\mu)\}.
\]

The morphisms in $G_\Lambda$ are triples $(\lambda z, d(\lambda) - d(\mu), \mu z) \in \partial \Lambda \times \ZZ^k \times\partial \Lambda$, where $(\lambda, \mu) \in \Lambda *_s \Lambda$ and $z \in s(\lambda)\partial \Lambda$. The unit space of $G_\Lambda$ is the set of morphisms $\{(x, 0, x) \mid x \in \partial E\}$ and we identify $(x, 0, x) \in \G_\Lambda^{(0)}$ with $x \in \partial \Lambda$.
The domain, codomain, and inversion, and composition maps are given by the following formulae. 
\begin{align*}
 \dom(x,m,y) = y, && \cod(x, m, y) = x, && (x, m, y)^{-1} = (y,-m,x), && (x, m, y)(y, l, z) = (x, m+l, z).
\end{align*}
To equip $G_\Lambda$ with a suitable topology, 
we define the following open sets for any pair $(\lambda, \mu) \in \Lambda *_s \Lambda$ and finite non-exhaustive subset $F \subseteq s(\lambda)\Lambda$, 
\begin{align*}
Z(\lambda*_s \mu) &:= \{(\lambda z, d(\lambda) - d(\mu), \mu z) \mid  z \in s(\lambda)\partial \Lambda \}, \\
Z(\lambda *_s \mu \setminus F) &:= Z(\lambda *_s \mu) \setminus \bigcup_{\phi \in F} Z(\lambda\phi *_s \mu\phi).
\end{align*}
The collection of sets $Z(\lambda *_s \mu \setminus F)$ forms a basis of compact open bisections for a Hausdorff topology on $G_\Lambda$, making it an ample groupoid. The continuous functor $\varphi: \G_\Lambda \to \ZZ^k$ sending $(x, m, y) \mapsto m$ makes $\G_\Lambda$ a $\ZZ^k$-graded groupoid (and hence its Steinberg algebra is $\ZZ^k$-graded).

We refer to \cite[Definition 3.1, Theorem 3.7]{CP} for a fully 
algebraic definition of the Kumjian-Pask algebra $\KP_R(\Lambda)$ of a 
finitely aligned higher-rank graph $\Lambda$, with coefficients in a unital 
commutative ring $R$. For our purposes it suffices to say that $\KP_R(\Lambda)$ 
is a $\ZZ^k$-graded $R$-algebra that is graded isomorphic to $A_R(G_\Lambda)$. 
We refer to \cite[Theorem 5.4]{CP} for the details on that isomorphism.

\begin{proposition}
\label{prop:grrfns}
Let $\Lambda$ be a finitely aligned $k$-graph.   
\begin{enumerate}[\upshape(1)]
\item\label{it1:grrfns} If $G_{\Lambda}$ is strongly ${\ZZ}^k$-graded, then $\Lambda$ has no sources.

\smallskip

\item\label{it2:grrfns}  If $G_{\Lambda}$ is strongly ${\ZZ}^k$-graded, then $\Lambda$ is row finite.
\end{enumerate}
\end{proposition}

\begin{proof}
For \eqref{it1:grrfns}, we prove the contrapositive.  Suppose $v \in \Lambda$ is a source.  
Then we apply Lemma~\ref{lem:sources} to get $x\in \partial \Lambda$ such that $d(x)_i < \infty$ for some $i \in \{1, \dots, k\}$.  Fix $m >d(x)_i$.
In the groupoid $\G_\Lambda$ we have 
\[\dom^{-1}(x) \cap (G_{\Lambda})_m = \emptyset\] 
and hence $G_{\Lambda}$ is not strongly graded. 

For \eqref{it2:grrfns} first apply part \eqref{it1:grrfns} and assume $\Lambda$ has no sources.  
By way of contradiction, assume $\Lambda$ is not row-finite.  Let $x$ be as in Lemma~\ref{lem:hardone}.
Then in the groupoid $\dom^{-1}(x) \cap (G_{\Lambda})_i = \emptyset$ which is a contradiction.
\end{proof}

To characterise strong grading on Kumjian-Pask algebras, we define a condition that generalises Condition (Y) from Definition \ref{Y}.

\begin{definition}
	Let $\Lambda$ be a $k$-graph. We say that $\Lambda$ satisfies \textit{Condition (Y)} if for every $m \in \NN^k$ and every infinite path $x \in \Lambda^\infty$, there exists some $n \in \NN^k$ and some path $\beta \in \Lambda$ such that $s(\beta) = x(n)$ and $d(\beta) - n = m$.
\end{definition}

\begin{example}\cite[Examples 2.2]{RSY03}
	Consider the $2$-graph $\Lambda$ with $1$-skeleton
	
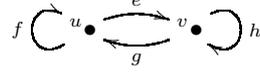
\[
	\xymatrix{^u\bullet\ar@(dl,ul)@[red]^f \ar@/^/@[blue][r]^e & ^v\bullet \ar@(ur,dr)@[red]^h \ar@/^/@[blue][l]^g}
	\]
	where blue edges ($e$ and $g$) have degree $e_1 = (1,0)$ and red edges ($f$ and $h$) have degree $e_2 = (0,1)$.
	The only infinite path $x$ with $r(x) = u$ is the following. 
	
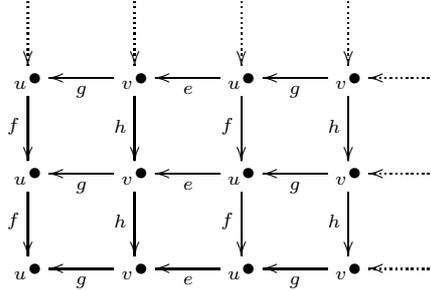
\[
	\xymatrix{ \ar@{.>}@[red][d] & \ar@{.>}@[red][d]& \ar@{.>}@[red][d] &  \ar@{.>}@[red][d] \\
		_u\bullet \ar@[red][d]_f & _v\bullet \ar@[red][d]_h \ar@[blue][l]^g& _u\bullet \ar@[red][d]_f \ar@[blue][l]^e &  _v\bullet \ar@[red][d]_h \ar@[blue][l]^g & \ar@[blue]@{.>}[l]\\
		_u\bullet \ar@[red][d]_f & _v\bullet \ar@[red][d]_h \ar@[blue][l]^g& _u\bullet \ar@[red][d]_f \ar@[blue][l]^e &  _v\bullet \ar@[red][d]_h \ar@[blue][l]^g & \ar@[blue]@{.>}[l]\\
		_u\bullet & _v\bullet \ar@[blue][l]^g &  _u\bullet \ar@[blue][l]^e &   _v\bullet \ar@[blue][l]^g &\ar@[blue]@{.>}[l] }
	\]
	The only infinite path $y$ with $r(y) = v$ is just $y = x(e_1, \infty) = gx$, which can be visualised by removing the first column in the diagram above. One can see that $\Lambda$ is row-finite without sources, and that it satisfies Condition (Y). For example, if $m \in \NN^k$ and $m_1$ is even, then $s(x(0, m)) = x(0) = u$ and $d(x(0,m)) = m$. If $m_1$ is odd, then $s(x(e_1, m+2e_1)) = s(x(e_1)) = v$, and $d(x(e_1, m+2e_1)) - e_1 = m$. A similar check works for the other infinite path $y$.
	
	There are already hints that the boundary path groupoid $G_\Lambda$, whose unit space is  $\partial \Lambda = \{x, y\}$, is strongly graded. For example, if $m = (-1,1)$, to prove that $(x,0,x) \in (G_\Lambda)_m(G_\Lambda)_{-m}$, we can write \[(x,0,x) = (fx,(-1,1),gx)(gx,(1,-1),fx).\]
	This idea is generalised and made precise in the following theorem.
\end{example}

\begin{theorem}\label{thm:KP}
	Let $\Lambda$ be a $k$-graph, and $R$ a unital commutative ring. The Kumjian-Pask algebra $\KP_R(\Lambda)$ is strongly $\ZZ^k$-graded if and only if $\Lambda$ is row-finite, has no sources, and satisfies Condition (Y).	
\end{theorem}
\begin{proof}
	We prove that $G_\Lambda$ is strongly graded if and only if $\Lambda$ satisfies the hypotheses. Assume that $G_\Lambda$ is strongly $\ZZ$-graded. By Proposition \ref{prop:grrfns}, $\Lambda$ is row-finite and has no sources, so $\partial \Lambda = \Lambda^\infty$. Now let $m \in \NN^k$, and $x \in \Lambda^\infty$. Then we can factor $x$ in $G_\Lambda$ as $x = (x,0,x) = (x, -m, y)(y, m, x)$ for some $y \in \Lambda^\infty$. This implies that for some $n \in \NN^k$, we have $x(n,\infty) = y(n+m, \infty)$. Letting $\beta = y(0, n+m)$, we have $s(\beta) = x(n)$ and $d(\beta) - n = m$. Therefore $\Lambda$ satisfies Condition (Y).
	
	Now suppose that $\Lambda$ is row-finite and has no sources (so $\partial \Lambda = \Lambda^\infty$), and satisfies Condition (Y). Let $x \in \Lambda^\infty$. For arbitrary $m \in \ZZ^k$, we aim to write
	\[
	(x,0,x) = (x, m, y)(y, -m, x), 
	\]
	where $(x, m, y), (y, -m, x) \in G_\Lambda$. If $m \ge 0$ then this is easy: let $y = x(m,\infty)$. If $m \le 0$ then Condition~(Y) implies there is a path $\beta \in \Lambda$ with $s(\beta) = x(n)$ and $d(\beta) - n = -m$, so we let $y = \beta x(n, \infty)$. There is a third case, where some $m_i$ are positive and others are negative. Let $d \in \NN^k, \ d_i := \underset{1 \le j \le k}{\max}\{m_j\}$, which ensures $-m+d \in \NN^k$. Then $x' = x(d, \infty) \in \Lambda^\infty$, so we can apply the hypothesis to $-m+d$ and $x'$. This provides some $n \in \NN^k$ and $\beta \in \Lambda$ with $s(\beta) = x'(n) = x(d+n)$ and $d(\beta) - n = -m + d$. It follows that
	\[
	(x,0,x) =\big (x, m, \beta x(d+n, \infty)\big)\big(\beta x(d+n, \infty),-m, x\big).
	\]
	Therefore, $G_\Lambda$ is strongly graded.
\end{proof}

\begin{remark}
We can recover our Leavitt path algebra Theorem~\ref{thm:LPA} as a special case of
Theorem~\ref{thm:KP}. Strictly speaking, however, we have assumed that the higher-rank graphs in Section \ref{KPA} are countable, and have made no such restrictions on the graphs in Section \ref{LPA}.
\end{remark}

\subsection{The Steinberg algebra of a transformation groupoid}

Here, we discuss the groupoid associated to a partial action of a discrete group on a topological space. This groupoid was defined by Abadie in \cite{abadie}, and it has been studied recently in the context of partial skew group rings \cite{BG,hli}. We show that the Steinberg algebra of this groupoid is strongly $\Gamma$-graded if and only if the partial action is a global action.

Let $\Gamma$ be a discrete group with identity $\ep$, and $X$ a topological space. A \textit{partial action} of $\Gamma$ on $X$, is a pair $\theta =(\{X_\gamma\}_{\gamma \in \Gamma}, \{\theta_\gamma \}_{\gamma \in \Gamma})$, where:
\medskip
\begin{enumerate}[({P}1)]
	\item Each $X_\gamma$ is open in $X$, and each $\theta_\gamma: X_{\gamma^{-1}} \to X_\gamma$ is a homeomorphism;
	
	\smallskip

	\item $X_\ep = X$ and $\theta_{\gamma\delta}$ is an extension of $\theta_\gamma \theta_\delta$ for every $\gamma, \delta \in \Gamma$.
\end{enumerate}
\medskip
We say that the partial action is a \textit{global action} if $X_\gamma = X$ for every $\gamma \in \Gamma$; this corresponds with the usual definition of a discrete group acting continuously on a space.

The \textit{transformation groupoid} $\Gamma \rtimes_\theta X$ of the partial action $\theta$ is a $\Gamma$-graded groupoid which has $X$ as its unit space.  We define the transformation groupoid as follows: 
\[
\Gamma \rtimes_\theta X := \{(x, \gamma, y) \mid y \in X_{\gamma^{-1}},\ x = \theta_{\gamma}(y) \},
\]
\begin{align*}
\dom(x,\gamma, y) = y, && \cod(x,\gamma,y) = x, && (x, \gamma, y)(y, \delta, z) = (x, \gamma \delta, z), && (x, \gamma, y)^{-1} = (y, \gamma^{-1}, x).
\end{align*}
The unit space is $\{(x, \ep, x)\mid  x \in X \}$; we identify it with $X$ and give it the same topology as $X$. If $X$ is locally compact, Hausdorff, and totally disconnected, the transformation groupoid is ample. On $\Gamma \rtimes_\theta X$ we take the topology inherited from $X \times \Gamma \times X$. There is some redundancy in our notation, since in the expression $(x, \gamma, y)$ the element $x$ is uniquely determined by $\gamma$ and $y$, but this notation makes composition and inversion easy to visualise and causes no inconsistencies (see \cite[p. 1042]{abadie} for further details). There is a natural $\Gamma$-grading on $\Gamma \rtimes_\theta X$ specified by the continuous functor $\Gamma \rtimes_\theta X \to \Gamma$, $(x, \gamma, y) \mapsto \gamma$. Transformation groupoids have the property that the $\ep$-component is equal to the unit space.

\begin{proposition} \label{paction}
	Let $\theta$ be a partial group action of a discrete group $\Gamma$ on a totally disconnected, locally compact Hausdorff space. Then the following are equivalent:
	\begin{enumerate}[\upshape(1)]
	\item $\theta$ is a global action;
	
	\smallskip

	\item $\Gamma \rtimes_\theta X$ is strongly graded;
	
	\smallskip

	\item $A_R(\Gamma \rtimes_\theta X)$ is strongly graded.
	\end{enumerate}
\end{proposition}

\begin{proof}
		$(1)\Rightarrow(2)$ If $\theta$ is a global action, then for every $\gamma \in \Gamma$ and $x \in X$ we have $(\theta_{\gamma}(x), \gamma, x) \in \Gamma \rtimes_\theta X$, so $x \in \dom((\Gamma \rtimes_\theta X)_\gamma)$. By Lemma \ref{stg} (3), $\Gamma \rtimes_\theta X$ is strongly graded.
		
		\smallskip
		
		$(2)\Rightarrow(1)$ Suppose $\Gamma \rtimes_\theta X$ is strongly graded. By Lemma \ref{stg} (3), for all $\gamma \in \Gamma$ and $x \in X$ we have $x \in \dom((\Gamma \rtimes_\theta X)_\gamma)$.  This implies $x \in X_{\gamma^{-1}}$ for all $x \in X$ and $\gamma \in \Gamma$. Therefore $\theta$ is a global action.
		
		\smallskip
		
		$(2) \Leftrightarrow (3)$ This is Theorem \ref{strong}.
\end{proof}

\section*{Acknowledgements}
Lisa Clark acknowledges Marsden grant 15-UOO-071 from the Royal Society of New Zealand.

Roozbeh Hazrat acknowledges the Australian Research Council grant DP160101481. A part of this work was done at the University of M\"unster, where he was a Humboldt Fellow.

Simon Rigby acknowledges the National Research Foundation of South Africa and his masters supervisor, Juana Sanchez-Ortega. A part of this work was done at the University of Wisconsin-Madison, where he was a visiting student.

The authors would like to thank Rabeya Basu for organising a very informative CIMPA workshop in Pune, June 2017, where the initial discussion of this project took place. They would also like to thank the referee for a careful reading of the paper and for comments which resulted in substantial improvement of the article.

\end{document}